\documentclass[11pt]{amsart}
\usepackage{verbatim, latexsym, amssymb, amsmath,color}
\usepackage{epsfig}
\usepackage{hyperref}
\usepackage{float}
\usepackage{graphicx}
\usepackage{subfig}
\usepackage{mathtools}
\usepackage{setspace}
\newtheorem{thm}{Theorem}[section]
\newtheorem{lemm}[thm]{Lemma}
\newtheorem{cor}[thm]{Corollary}
\newtheorem{defi}[thm]{Definition}
\newtheorem{prop}[thm]{Proposition}

\theoremstyle{remark}
\newtheorem{rmk}[thm]{Remark}

\theoremstyle{definition}

\numberwithin{equation}{subsection}

\title[Widths of spheres associated to the distance function]{The min-max width of spheres associated to the distance function}
\author{Rafael Montezuma and Idalina Ribeiro}

\address{Departamento de Matemática, UFC \\ Fortaleza/CE 60455-760 Brazil}
\email{montezuma@mat.ufc.br}
\email{idalinasilva1a@gmail.com}

\thanks{R.M. was supported by Instituto Serrapilheira grant ``New perspectives of the min-max theory for the area functional".}

\begin{document}

\begin{abstract}
	What one obtains when the min-max methods for the distance function are applied on the space of pairs of points of a Riemannian two-sphere? This question is studied in details in the present article. We show that the associated min-max width do not always coincide with half of the length of a simple closed geodesic which is the union of two minimizing geodesics with the same endpoints. Therefore, it is a new geometric invariant. We study the structure of the set of minimizing geodesics joining a pair of points realizing the width, and relationships between this invariant and the diameter. The extrinsic case of an embedded Riemannian sphere is also considered.
\end{abstract}

\maketitle

\section{Introduction}

The min-max widths $\{\omega_k(M^{n+1}, g)\}$, $k\in \mathbb{N}$, are a sequence of geometric invariants of an $(n+1)$-dimensional closed Riemannian manifold, associated to special critical values of the $n$-dimensional volume functional. These quantities are related to several deep application in min-max theory and, more generally, in differential geometry in the past decade. When $3\leq n+1 \leq 7$, each $\omega_k(M, g)$ is realized by a weighted $n$-dimensional volume of smooth closed embedded minimal hypersurfaces. Deeper properties of this sequence of numbers, associated hypersurfaces, and relationships with other invariants are part of a very active research field. We refer the reader to \cite{MarNev}, \cite{LioMarNev}, \cite{MarNevDuke}, \cite{AmbMon}, \cite{ChoMan1}, \cite{Song}, \cite{Zhou}, and references therein. 

When $n+1 =2$, the widths are achieved by the weighted lengths of immersed closed geodesics (\cite{ChoMan}, \cite{Pitts}, \cite{SarStr}, \cite{CalabiCao}, \cite{Aiex}, \cite{Don}, \cite{DM}).

In a recent article, \cite{AMS}, Ambrozio, Santos, and the first named author introduced a new Morse-Lusternik-Schinirelmann theory for the distance function defined on the space of pairs of points of an embedded circle in a Riemannian manifold. The associated min-max quantity generalizes the classical notion of width of convex plane curves. It is also interesting to analyze the structure of the set of minimizing geodesics joining a pair of points that is a critical point of the distance function realizing the min-max width. The natural comparison between the new invariant and the extrinsic diameter of the curve allows for connections with the theory of curves of constant width. Examples in \cite{AMS} explained that the new invariant and the relative first Almgren-Pitts min-max width, $\omega_1(\mathbb{D}, g)$, of a Riemannian two-dimensional disk with convex boundary do not always coincide. More precisely, the minimizing geodesics joining the points in a critical pair realizing the width of the boundary curve associated to the distance function are not necessarily perpendicular to the boundary.

A natural question arises. What one obtains when the min-max methods developed in \cite{AMS} for the distance function are applied on the space of pairs of points of a Riemannian sphere? We investigate this question in details. 

\subsection{The width of an embedded sphere}\label{subsect-width-introd} Let $\Sigma$ be a smoothly embedded $2$-sphere in a complete Riemannian manifold $(M^n,g)$.  The case in which $M$ is a $2$-sphere itself and the embedding is a diffeomorphism is also considered here, i.e., it is admissible to have $(M^n, g) = (S^2, g)$ and $\Sigma = M$. We use $x \in S^2$ and $x\in \Sigma$ interchangeably to refer to the points of $S^2$ and $\Sigma$ that are identified by the embedding. In this section, we define the basic notions of sweepouts in the space of pairs of points $\mathcal{P}_{\Sigma}$, and the distance related min-max width of $\Sigma$. Let us begin with the basic definition of the space of pairs of points of the embedded sphere $\Sigma \subset M$,
$$\mathcal{P}_{\Sigma}: = \{\{i(p), i(q)\} \subset \Sigma: p, q \in S^2\},$$
where $i: S^2\to M$ is an embedding with $i(S^2)=\Sigma$. Notice that $\mathcal{P}_{\Sigma}$ is the set of subsets of $\Sigma$ with one or two points. This space, with its natural topology, is homeomorphic to the complex projective plane $\mathbb{CP}^2$. This fact and other properties of the space of pairs of points are proved in Section \ref{sect-details-part-I}.

In order to properly define the class of sweepouts of $\Sigma$ that is considered in this work, fix $x_0 \in \Sigma$ and define the map
$\Phi_0 : S^2 \rightarrow \mathcal{P}_{\Sigma}$, $\Phi_0(x):= \{x, x_0\}$, for $x\in S^2$, where $x\in S^2$ is identified with $x\in \Sigma$ by the embedding. This map is not homotopic to a constant map, as proven in Section \ref{subsect-std-sweep}. Therefore, its homotopy class is a natural class of maps to perform min-max.

\begin{defi}\label{defi-sweepout}
A continuous map $\phi : S^2 \rightarrow \mathcal{P}_{\Sigma}$ is a sweepout of $\Sigma$ if there exist continuous maps $\phi_1, \phi_2 : S^2 \rightarrow \Sigma$ such that $\phi(x) = \{\phi_1(x), \phi_2(x)\}$, for every $x\in S^2$, and $\phi$ is homotopic to the standard map $\Phi_0$ defined above.
\end{defi}

\begin{defi}\label{defi-width}
Let $\Sigma$ be a smoothly embedded $2$-sphere in a complete Riemannian manifold $(M^n,g)$. The min-max width of $\Sigma$ is defined by 
$$
W_d(\Sigma) := \inf_{\phi} \max_{x \in S^2} dist(\phi_1(x), \phi_2(x)),  
$$
where the infimum is taken over the set of all sweepouts $\phi=\{\phi_1, \phi_2\}$ of $\Sigma$, and $dist(\cdot, \cdot)$ represents the canonical distance function induced by $g$ on $M$.
\end{defi}

The positivity of the width depends not only on the fact that sweepouts are not homotopic to a constant map. We verify that maps from $S^2$ to $\mathcal{P}_{\Sigma}$ whose image is contained in the set of trivial pairs $\{x\} \in \mathcal{P}_{\Sigma}$ are not homotopic to $\Phi_0$. Further details are included in Section \ref{subsect-not-sweepout}. 

The definitions above are motivated by the Morse-Lusternik-Schnirelmann theory about the number of critical points of a smooth function on $\mathbb{CP}^2$, as in the case of the width of embedded curves considered in \cite{AMS}. One important difference in the present article is that one must work with $2$-parameter sweepouts. The increased number of parameters poses new difficulties in the construction of sweepouts. Another aspect of the present article is that the intrinsic case $(S^2, g)$ is already non-trivial. We still need to overcome the difficulties imposed by the non-local character of the distance function and the lack of smoothness of this function at some pairs in $\mathcal{P}_{\Sigma}$. The notion of regular points of the distance is as follows.

\begin{defi}\label{defi-regular-point}
A point $\{x,y\} \in \mathcal{P}_{\Sigma}$ is a regular point of the Riemannian distance function if $x\neq y$ and there exists a vector $(w_1, w_2) \in T_x\Sigma \times T_y\Sigma$ satisfying the following property: whenever $\gamma : [0, a] \rightarrow M$ is minimizing geodesic (i.e., it realizes the distance between its endpoints), parametrized by arc-length with extremities $\gamma(0)=x$ and $\gamma(a)=y$, one has that
$$
g( (w_1, w_2), \nu_{\gamma} ):= g(w_2, \gamma^{\prime}(a)) - g(w_1, \gamma^{\prime}(0)) < 0,
$$
where $\nu_{\gamma} = (-\gamma^{\prime}(0), \gamma^{\prime}(a))$ is the vector in $T_xM \times T_yM$ that collects the outward pointing unit tangent vectors of $\gamma$ at the endpoints. 
\end{defi}

Notice that we are using the same notation $g$ for the inner products induced on $T_xM\times T_yM$, for $x, y \in M$, i.e., $g((w_1, w_2), (v_1, v_2)):= g(w_1, v_1) + g(w_2, v_2)$. The first variation formula implies that a variation of a minimizing geodesic $\gamma$ as in Definition \ref{defi-regular-point} in the direction of a vector field $X$ along $\gamma$ with $X(0) = w_1$ and $X(a)=w_2$ has negative first derivative of the length.

\begin{defi}\label{defi-critical-point}
A point $\{x,y\} \in \mathcal{P}_{\Sigma}$ is called a critical point of the distance function if it is not a regular point in the sense of Definition \ref{defi-regular-point}.
\end{defi}

Critical points $\{x,y\}$ are characterized in terms of the structure of the set of minimizing geodesics joining $x$ and $y$ using the following terminology.

\begin{defi}\label{defi-simultaneously-stationary}
Let $\{x, y\} \in \mathcal{P}_{\Sigma}$, and $\gamma_1, \ldots, \gamma_k : [0, a] \rightarrow M$ be minimizing geodesics parametrized by arc-length with endpoints $x = \gamma_i(0)$ and $y = \gamma_i(a)$, for every $1\leq i \leq k$. The set $\{\gamma_1, \ldots, \gamma_k\}$ is called simultaneously stationary with endpoints $x$ and $y$ if there exist constants $0<c_i \leq1$, $1\leq i \leq k$, such that $\sum_{i=1}^k c_i =1$ and $\sum_{i=1}^k c_i \nu_{\gamma_i}^T$ is the zero vector in $T_{x}\Sigma \times T_y \Sigma$, where $\nu_{\gamma_i} = (-\gamma_i^{\prime}(0), \gamma_i^{\prime}(a))$ is the collection of the outward pointing unit tangent vectors of $\gamma_i$ at the endpoints of this geodesic, and $\nu_{\gamma_i}^T = (-\gamma_i^{\prime}(0)^T, \gamma_i^{\prime}(a)^T)$ denotes the orthogonal projections of those vector on $T\Sigma$.   
\end{defi}

Definition \ref{defi-simultaneously-stationary} includes the case $k=1$, which corresponds to the case where there exists a minimizing geodesic joining $x$ and $y$ which is orthogonal to $\Sigma$ at these endpoints, i.e., a free boundary geodesic which is also minimizing. In the intrinsic case, $\Sigma = M$, a set of simultaneously stationary geodesics contains at least two curves, and a set $\{\gamma_1, \gamma_2\}$ with two simultaneously stationary geodesics is such that $\gamma_1\cup \gamma_2$ is a closed geodesic of $\Sigma$.

The following proposition explains the relationship between the critical points and the existence of simultaneously stationary geodesics.

\begin{prop}\label{prop-characterization-critical-simult-st}
Let $\Sigma$ be a smoothly embedded $2$-sphere in a complete Riemannian manifold $(M^n,g)$, and $\{x, y\} \in \mathcal{P}_{\Sigma}$. If there exists a set $\{\gamma_1, \ldots, \gamma_k\}$ of simultaneously stationary geodesics with endpoints $x$ and $y$, then $\{x, y\}$ is a critical point of the distance function. Conversely, if $\{x, y\}$ is a critical point of the distance, then, there exists an integer number $k \leq 5$, and a set of $k$ simultaneously stationary minimizing geodesics with endpoints $x$ and $y$.  
\end{prop}

Proposition \ref{prop-characterization-critical-simult-st} is proved in Section \ref{subsect-regular-and-critical}. The first main result of this work is the fact that $W_d(\Sigma)$ is realized as the distance between a pair of points of $\Sigma$ which form a critical point of the Riemannian distance function. 

\begin{thm}\label{thm-W_d-critical}
Let $\Sigma$ be a smoothly embedded $2$-sphere in a complete Riemannian manifold $(M^n,g)$. Then, the min-max width $W_d(\Sigma)$ is a strictly positive critical value of the distance function. In particular, there exists a collection of simultaneously stationary geodesics with at most five elements joining points at distance $W_d(\Sigma)$.
\end{thm}

In the intrinsic case $S^2=\Sigma = M$, the new invariant detects a geometric number that can be different from quantities associated to the construction of closed geodesics. For the round metric on $S^2$, equipped with the intrinsic distance, the only pairs of critical points are those of the form $\{x, -x\}$, made of two antipodal points, and the trivial pairs $\{x\}$, where $x=y$. In this case, every non-trivial critical point $\{x, -x\}$ bounds pairs of simultaneously stationary minimizing geodesics whose union is a great circle. On the other hand, there are Riemannian metrics on $S^2$ containing critical points which do not bound two simultaneously stationary geodesics, see Section \ref{subsect-gluing-2-triangles}. These critical points can even be the only ones that realize $W_d(S^2, g)$.

The new width also detects an invariant different from those that are usually built to detect free boundary geodesics in the extrinsic case, $\Sigma \subset M$. This can be observed on rotationally symmetric metrics on $\mathbb{R}^3$ obtained by a perturbation of the euclidean metric near the origin in such a way that the portions of geodesics through the origin satisfy the free boundary condition on spheres centered at this point, but they do not minimize the length among curves with the same endpoints. It is even possible to obtain positively curved examples with similar properties, see Section \ref{exam-example-ellipsoid}.

\subsection{The intrinsic width of spheres whose closed geodesics are all unstable}\label{subsect-intrinsic-thmb}

In this section we consider the intrinsic case $(S^2, g)$ only. We analyze further properties of the structure of the set of minimizing geodesics connecting the points $x$ and $y$ of a critical point such that $dist(x, y) = W_d(S^2, g)$. We have already pointed out that $\{x, y\}$ bounds at least two minimizing geodesics, but it is not always true that it bounds a pair of simultaneously stationary minimizing geodesics. Under a natural assumption, we are able to refine Proposition \ref{prop-characterization-critical-simult-st}.

In order to state our theorem precisely, we recall the notions of stability and Morse index of closed geodesics. Let $\gamma: [0, a] \rightarrow S^2$ be a closed geodesic parametrized by arc-length with respect to the Riemannian metric $g$. Given a vector field $V$ along $\gamma$ and normal to this curve, define
\begin{equation}\label{defi-quadratic-form-second-variation}
\mathcal{Q}(V, V):= \int_0^a \left(|V^{\prime}(s)|^2 - K(\gamma(s))|V(s)|^2\right) ds,
\end{equation}
where $V^{\prime}(s)$ is the covariant derivative of $V$, and $K(\gamma(s))$ represents the Gauss curvature of $(S^2,g)$ at $\gamma(s)$. Letting $\{\gamma_t\}$ be a smooth variation of $\gamma_0=\gamma$ by smooth closed curves with variational vector field $V$, the second derivative of the length $L(\gamma_t)$ equals $\mathcal{Q}(V,V)$. The \textit{Morse index of} $\gamma$ is the index of the quadratic form $\mathcal{Q}$, i.e., the maximal dimension of a vector space of vector fields normal to $\gamma$ on which $\mathcal{Q}$ is negative definite. The geodesic $\gamma$ is called \textit{stable} if its Morse index is equal to zero. Positively curved Riemannian spheres do not admit stable closed geodesics. 

The main result that concerns the intrinsic case is the following theorem.

\begin{thm}\label{thm-3-simult-st}
Let $(S^2,g)$ be a Riemannian sphere. Assume that $(S^2, g)$ contains no closed geodesics which are stable. Then, the min-max width $W_d(S^2,g)$ is the lowest strictly positive critical value of the distance function. In addition, if $\{x, y\}$ is a critical point with $dist(x,y) = W_d(S^2, g)$, then there exists an integer number $k \in \{2,3\}$, and a set of $k$ simultaneously stationary minimizing geodesics with endpoints $x$ and $y$.
\end{thm}

This result should be compared with part \textit{ii)} of Theorem B of \cite{AMS}. Indeed, in the setting of Theorem \ref{thm-3-simult-st}, assuming that $\{x,y\}$ is a critical point with $dist(x,y)=W_d(S^2,g)$, and that $\gamma_1, \gamma_2, \gamma_3 : [0, a] \rightarrow S^2$ are three different minimizing geodesics joining $\gamma_i(0)=x$ and $\gamma_i(a)=y$, and satisfying that the three angles determined by the $\gamma_i^{\prime}(0)$ at the origin of $T_xS^2$, and the angles determined by the vectors $-\gamma_i^{\prime}(a)$ at the origin of $T_yS^2$, all measure at most $\pi$, then we prove that $\{\gamma_1, \gamma_2, \gamma_3\}$ is simultaneously stationary with endpoints $x$ and $y$. See Proposition \ref{3-simult-strict}. In order to achieve this result, we adapt techniques from \cite{AMS} to produce $2$-parameter sweepouts as in Definition \ref{defi-sweepout}. In the present setting, we overcome an extra difficulty. Namely, we prove that there exists either a set of two minimizing geodesics which is simultaneously stationary with endpoints $x$ and $y$, or a set $\{\gamma_1, \gamma_2, \gamma_3\}$ of three minimizing geodesics satisfying the angular hypotheses above. Our strategy involves proving the following convex geometry lemma of independent interest.

\begin{lemm}\label{lemm-convex-geometry}
	Let $v_i$, $w_i \in \mathbb{R}^2$ be unit vectors and $\lambda_i \in (0,1)$, $i=1, 2, 3, 4$. Suppose that the vectors $v_1$, $v_2$, $v_3$, and $v_4$ are disposed in this order as one moves around the unit circle in the clockwise direction, while the vectors $w_1$, $w_2$, $w_3$, and $w_4$ appear in this order in the counterclockwise direction. In addition, assume that $\sum_{i=1}^4 \lambda_i =1$ and $\sum_{i=1}^4 \lambda_i v_i = 0 = \sum_{i=1}^4 \lambda_i w_i$.

	Then, there exists $A\subset \{1, 2, 3, 4\}$, $\#A =3$, such that the origin of $\mathbb{R}^2$ belongs to the triangles with vertices either in $\{v_i : i \in A\}$ or in $\{w_i : i \in A\}$. 
\end{lemm}

The hypothesis about non-existence of stable closed geodesics in Theorem \ref{thm-3-simult-st} cannot be removed, as one can see in Example \ref{exam-critical-value-below-width} in Section \ref{sect-intrinsic-thmb}. The Calabi-Croke sphere is an example of a configuration with three simultaneously stationary minimizing geodesics, as described in Theorem \ref{thm-3-simult-st}. This is the content of Section \ref{subsect-gluing-2-triangles}. In the case that the critical point realizing the width bounds a pair of minimizing geodesics that form a closed geodesic, we prove that the Morse index of this curve is equal to one.

\begin{thm}\label{thm-index-bound}
Let $(S^2,g)$ be a Riemannian sphere. Assume that $(S^2, g)$ contains no closed geodesics that are stable. If $\{x, y\}\subset S^2$ is a critical point of the distance function such that $dist(x,y) = W_d(S^2, g)$, and there exists two distinct minimizing geodesics $\gamma_1$ and $\gamma_2$ joining $x$ and $y$ such that $\gamma_1\cup\gamma_2=\gamma$ is a closed geodesic, then $\gamma$ has Morse index equal to one.
\end{thm}

A consequence of our results is the following corollary.

\begin{cor}
Let $(S^2,g)$ be a Riemannian sphere whose closed geodesics are all unstable. Then, there exists either two distinct minimizing geodesics $\gamma_1$ and $\gamma_2$ whose union $\gamma_1\cup\gamma_2$ is a closed geodesic of Morse index one, or there exists a set of three simultaneously stationary minimizing geodesics.
\end{cor}

Calabi and Cao \cite{CalabiCao} proved that the first Almgren-Pitts min-max width $\omega_1(S^2, g)$ of positively curved spheres are realized by the length of simple closed geodesics. This fact and the methods applied in the proof of Theorem \ref{thm-index-bound} imply the following comparison result.

\begin{thm}\label{thm-comparison}
Let $(S^2,g)$ be a positively curved Riemannian sphere. Then, the first min-max width of the Almgren-Pitts theory and the min-max width associated to the distance function are related by the inequality
$$
W_d(S^2,g) \leq \frac{\omega_1(S^2, g)}{2}.
$$
Moreover, if equality holds, then there are two distinct minimizing geodesics $\gamma_1$ and $\gamma_2$ of length $W_d(S^2,g)$ whose union is a closed geodesic of index one.
\end{thm}

\subsection{Intrinsic and extrinsic spheres of constant width}\label{subsect-constant-width}

The classical width of a compact region $\Omega \subset \mathbb{R}^n$ is defined as the shortest possible distance between two parallel hyperplanes $H_1$ and $H_2$ such that $\Omega$ is contained in the closure of the connected component of $\mathbb{R}^n\setminus (H_1\cup H_2)$ whose boundary is $H_1\cup H_2$. There is also a notion of width in a given direction.

\begin{defi}\label{defi-classical-width}
Let $\Omega$ be a compact region of $\mathbb{R}^n$, and $v\in \mathbb{R}^n$ be a unit vector. The width of $\Omega$ in the direction of $v$ is defined as the smallest possible distance $w(\Omega, v)$ between hyperplanes $H_i = \{x \in \mathbb{R}^n : \langle x, v\rangle =a_i \}$, $i=1,2$, such that $a_1\leq a_2$ and $\langle x, v\rangle \in [a_1, a_2]$, for all $x\in \Omega$. Using this terminology, the classical width of $\Omega$ is given by $w(\Omega):= \min \{w(\Omega, v) : v \in \mathbb{R}^{n}, |v|=1\}$.
\end{defi}

There is a rich literature on curves and higher dimensional regions of constant width, i.e., regions satisfying $w(\Omega, v) = w(\Omega)$, for all unit vectors $v \in \mathbb{R}^n$. See, for instance, \cite{BonFen}, \cite{ChaGro}, and \cite{MarMonOli}. 

The notion of width of embedded circles considered in \cite{AMS} coincides with the classical one in the case of convex plane curves. Let us briefly recall that definition. Given an embedded circle $\Gamma$ in a complete Riemannian manifold $(M^n,g)$, one uses $\mathcal{P}_{\Gamma}$ to denote the space of subsets of $\Gamma$ consisting of one or two points. Then, $\mathcal{P}_{\Gamma}$ is homeomorphic to the M\"obius band, and one defines a notion of \textit{sweepout of} $\Gamma$ as a continuous map $\phi : [0,1]\rightarrow \mathcal{P}_{\Gamma}$ for which there exist continuous maps $x, y : [0,1] \rightarrow \Gamma$ such that $\phi(t) = \{x(t), y(t)\}$, with $x(0)=y(0)$ and $x(1)=y(1)$. In addition, the points $x(t)$ and $y(t)$ bound an arc $C(t)$ of $\Gamma$ which varies continuously and such that $C(0)=\varnothing$ and $C(1)=\Gamma$. The min-max width of $\Gamma$ introduced in \cite{AMS} is defined by
$$
\mathcal{S}(\Gamma) = \inf \bigg\{  \max_{t\in [0,1]} dist(\phi(t)) : \phi \text{ is a sweepout}\bigg\}.
$$

The notion of plane curves of constant width, studied at least since Euler, is also generalized in the setting of \cite{AMS}, even regarding the co-dimension of the curve. More precisely, the generalization consists in considering embedded circles $\Gamma$ on which the min-max invariant $\mathcal{S}(\Gamma)$ equals the extrinsic diameter of $\Gamma$. Indeed, Theorem C of \cite{AMS} relates this notion to the property of every point in $\Gamma$ being diametrically distant from another point of this curve. Indeed, Reidemeister \cite{Rei} characterized the regions of constant width in Euclidean space (in all dimensions) by the property that every boundary point is the endpoint of a diameter. In \cite{AMS}, the authors show that other properties of curves of constant width generalize to the Riemannian setting.

The higher dimensional case of the classical width is also well studied, with several nice examples, interesting geometric inequalities, and open problems. See, for instance, Theorem 12.1.4 of \cite{MarMonOli} and its generalizations to higher dimensions in the form of relationships between the quermassintegrals. The Notes section of Chapter 14 in \cite{MarMonOli} contains discussions on open problems concerning the geometry of such objects.

In analogy with the case of embedded circles studied in \cite{AMS}, we observed that $W_d(\Sigma)$ coincides with the classical width, whenever $\Sigma$ is the boundary of a convex region in $\mathbb{R}^3$. We state this result precisely in what follows. The proofs of the results stated in this subsection are contained in Section \ref{sect-const-width-proofs}. 

\begin{prop}\label{prop-widths-coincide}
The width associated to the extrinsic Euclidean distance of a compact convex regular domain $\Omega\subset \mathbb{R}^3$ satisfies $W_d(\partial \Omega) = w(\Omega)$. 
\end{prop}

We also have a version of Reidemeister's characterization, and the uniqueness of centrally symmetric regions of generalized constant width associated to the new min-max width. We proved the following results.

\begin{thm}\label{thm-Reidemeister}
Let $\Sigma$ be a smoothly embedded $2$-sphere in a complete Riemannian manifold $(M^n,g)$, possibly $\Sigma = M^2$ (the intrinsic case). 
\begin{enumerate}
\item[(I)] If $W_d(\Sigma)$ equals the extrinsic diameter of $\Sigma$, then for each $x \in \Sigma$ there exists $y \in \Sigma$, such that $dist(x,y)= W_d(\Sigma)=diam(\Sigma)$. 
\item[(II)] Suppose that there exists a continuous map $\Psi : \Sigma \rightarrow \Sigma$ such that $dist(x, \Psi(x)) = diam(\Sigma)$, for every $x\in \Sigma$. Then, $W_d(\Sigma) = diam(\Sigma)$.
\end{enumerate}
Moreover, in the case $\Sigma = M^2$, $W_d(\Sigma)$ coincides with the diameter of $(\Sigma, g)$ if, and only if, for each $x \in \Sigma$, $dist(x,y)= diam(\Sigma)$ for some $y \in \Sigma$.
\end{thm}

\begin{rmk}\label{rmk-fixedpoints}
  The proof of part (II) of Theorem \ref{thm-Reidemeister} is based on the fact that every sweepout passes through a pair of the form $\{y, \Psi(y)\}$, for some $y \in \Sigma$. The argument provided in the proof of the theorem, in Section \ref{sect-const-width-proofs}, works for every continuous map $\Psi$ without fixed points.  
\end{rmk}

In regard of the centrally symmetric regions of constant width, we proved the following result.

\begin{thm}\label{thm-involution}
Let $(M^3, g)$ be a complete Riemannian manifold, and $\Omega$ be a subset of $M$ with smooth boundary and diffeomorphic to a $3$-dimensional ball. Suppose, in addition, that the following hypotheses hold: 
\begin{enumerate}
\item[(i)] for every $x, y \in \Omega$, there exists a unique minimizing geodesic joining $x$ and $y$, and this geodesic is contained in $\Omega$;
\item[(ii)] for every geodesic $\gamma : [0, \infty) \rightarrow M$ emanating from a point $\gamma(0) \in \Omega$, there exists $t>0$ such that $\gamma(t)$ does not belong to $\Omega$;
\item[(iii)] there exists an isometric involution $\Psi : \Omega \rightarrow \Omega$ without boundary fixed points, i.e., $\Psi(x)\neq x$, for every $x\in \partial \Omega$.
\item[(iv)] any non-minimizing geodesic joining points of $\Sigma = \partial \Omega$ has length strictly bigger than $2 \cdot diam(\Omega)$ (twice the diameter of $\Omega$). 
\end{enumerate}
If $W_d(\partial \Omega) = diam(\partial\Omega)$, then $\Omega$ is a geodesic ball of $(M^3, g)$ and $\Psi$ is the symmetry with respect to the center $p$ of this ball, i.e., $\Psi(\sigma(s)) = \sigma(-s)$ whenever $\sigma$ is a geodesic of $(M^3, g)$ satisfying $\sigma(0)=p$.
\end{thm}

\begin{rmk}
Hypotheses (i), (ii), and (iv) hold if $\Omega$ is contained in a sufficiently small geodesic ball. It is interesting to compare (iv) and the property that any two points of the boundary of the region are joined by a unique (not necessarily minimizing) geodesic, as assumed on Theorem F of \cite{AMS}.
\end{rmk}

This result generalizes the well-know fact that the ball is the only centrally symmetric body of constant width, see for instance Theorem 3.5.1 of \cite{MarMonOli}. Notice that every geodesic sphere of sufficiently small radius in a Riemannian three-dimensional manifold satisfies $W_d(\partial \Omega) = diam(\partial \Omega)$, even if the symmetry with respect to its center is not an isometry. In this case, use the extrinsic distance induced by $(M^3, g)$. Indeed, this is a consequence of Part (II) of Theorem \ref{thm-Reidemeister} using the map that associates to each point of the geodesic sphere the only other point in this surface at distance $diam(\partial \Omega)$.

\subsection*{Organization:}

The proofs of the basic results involving critical points, which are the content of Section \ref{subsect-width-introd}, are included in Section \ref{subsect-regular-and-critical}. Section \ref{sect-details-part-I} also includes the details about the topology of the space of pairs of points of the embedded sphere, and the sweepouts of this space. The results of Section \ref{subsect-intrinsic-thmb}, about the min-max width of $2$-spheres whose closed geodesics are unstable, are proved in Section \ref{sect-intrinsic-thmb}. Section \ref{sect-const-width-proofs} contains the proofs of the facts about the generalization of the notion of constant width surfaces.

\subsection*{Acknowledgements:} The authors are grateful to Instituto de Matem\'atica Pura e Aplicada (IMPA) and its staff for the hospitality during the summer program of 2024, where part of this work was carried out. R.M. would like to thank Lucas Ambrozio for insightful conversations. The authors would like to thank the anonymous referees for a very careful reading of the manuscript and for many helpful comments and suggestions. 

\section{The space of pairs of points, sweepouts, and critical points}\label{sect-details-part-I}

This section includes the details about the topology of the space of pairs of points of an embedded (or intrinsic) sphere, and the sweepouts of this space. In Section \ref{subsect-topology-space-pairs}, we explain the equivalence between $\mathcal{P}_{\Sigma}$, the space of pairs, and the complex projective space. In Sections \ref{subsect-std-sweep} and \ref{subsect-not-sweepout} we prove the basic facts about sweepouts stated in Section \ref{subsect-width-introd}. The proofs of the basic results involving critical points are included in Section \ref{subsect-regular-and-critical}. It includes the proof of the Min-max Theorem \ref{thm-W_d-critical}.

\subsection{The topology of the space of pairs of points of \texorpdfstring{$S^2$}{S2}}\label{subsect-topology-space-pairs}

Consider the complex projective plane $\mathbb{CP}^2$ identified with the space of equivalence classes of non-zero polynomials with complex coefficients and degree at most two by the equivalence relation of complex multiplication, i.e., $[a_0+a_1 z + a_2 z^2]=[(\lambda a_0)+(\lambda a_1) z + (\lambda a_2) z^2]$, for every $\lambda \in \mathbb{C}\setminus \{0\}$. Assume that $S^2$ is identified with the Riemann sphere $\mathbb{C}\cup \{\infty\}$, and let
\begin{equation}
\xi : \mathcal{P}_{S^2} \rightarrow \mathbb{CP}^2, \quad \xi(\{x_1, x_2\}) := [(z+x_1)(z+x_2)],
\end{equation}
where $z+x_i = 1$, if $x_i = \infty$. The space $\mathcal{P}_{S^2}$ is equipped with the quotient topology of $S^2\times S^2$ under the equivalence relation $(x_1, x_2)\sim (x_2, x_1)$. The map $\xi$ is a homeomorphism. Indeed, its continuity is obvious whenever $x_1, x_2\neq \infty$. If one of the variables tends to the point at infinity, say $x_1$, the continuity follows from $[(z+x_1)(z+x_2)] = \left[\left(z/x_1+1\right)(z+x_2)\right]$. The case in which both variables $x_1$ and $x_2$ tend to infinity is similar. The inverse of $\xi$ consists in factoring out a polynomial of a given equivalence class, and pairing the negative of the roots (using the point at infinity once or twice depending on the degree of the polynomial). The roots are clearly independent of the choice of representative.

\subsection{The standard 2-sweepout}\label{subsect-std-sweep}

Next, we prove that the map $\Phi_0$ defined in the Introduction is non null-homotopic. Let us recall the definition. Fixed $x_0 \in S^2$, let $\Phi_0 : S^2 \rightarrow \mathcal{P}_{S^2}, \, \Phi_0(x):= \{x, x_0\}, \text{ for } x\in S^2$. In Definition \ref{defi-sweepout}, we used $\mathcal{P}_{\Sigma}$ as the target space of $\Phi_0$, but these spaces are equivalent.

It is convenient to consider the fiber bundle structure $S^1 \rightarrow S^5 \rightarrow \mathbb{CP}^2$ of the $5$-sphere over $\mathbb{CP}^2$. Namely, $S^5$ is the unit sphere in $\mathbb{C}^3$ and $\mathbb{CP}^2$ is viewed as the quotient space of $S^5$ under the equivalence relation $(z_0, z_1,z_2) \sim (\lambda z_0, \lambda z_1, \lambda z_2)$, for $\lambda \in S^1\subset \mathbb{C}$. The projection map is $p : S^5 \rightarrow \mathbb{CP}^2$ defined by $p(z_0, z_1, z_2) := [(z_0,z_1, z_2)]$, and the fibers are copies of $S^1$. We refer the reader to the discussion on fiber bundles of Section 4.2 in Hatcher's book \cite{Hat} for further details and properties of fiber bundles that are applied below.

Fix $b_0 \in \mathbb{CP}^2$ and $q_0 \in p^{-1}(b_0)$. We sometimes write $S^1 = p^{-1}(b_0)$ to represent the fiber over $b_0$. The homotopy long exact sequence of the pair $S^1 = p^{-1}(b_0) \subset S^5$ yields the exact sequence
\begin{align*}
\{0\}=\pi_2(S^5, q_0) \xrightarrow{j_{\ast}} \pi_2(S^5, S^1, q_0) \xrightarrow{\partial_{\ast}} \pi_1(S^1, q_0) \xrightarrow{i_{\ast}} \pi_1(S^5, q_0) = \{0\},
\end{align*}
where $i_{\ast}$ is induced by the inclusion of $S^1\subset S^5$, and $j_{\ast}$ is defined by viewing the maps $f : (\mathbb{D}, \partial\mathbb{D})\rightarrow (S^5, q_0)$, from the $2$-disk that map the boundary to the point $q_0$, as maps $f:(\mathbb{D}, \partial\mathbb{D}, y_0) \rightarrow (S^5, S^1, q_0)$, where $y_0$ is any base point at $\partial \mathbb{D}$. The map $\partial_{\ast}$ is induced by restriction to the boundary; a base point preserving map $f : (\mathbb{D}, \partial \mathbb{D}, y_0) \rightarrow (S^5, S^1,  q_0)$ is restricted to the boundary, yielding a base point map $f|_{\partial\mathbb{D}} : (\partial\mathbb{D}, y_0) \rightarrow (S^1, q_0)$ so that $\partial_{\ast}[f]:= [f|_{\partial\mathbb{D}}]$. In particular, $\partial_{\ast}$ is an isomorphism, and $\pi_2(S^5, S^1, q_0)$ is isomorphic to $\pi_1(S^1, q_0) = \mathbb{Z}$. 

The fiber bundle $S^1 \rightarrow S^5 \rightarrow \mathbb{CP}^2$, as any other fiber bundle, is such that the map induced by the projection $p_{\ast} : \pi_n(S^5, S^1, q_0) \rightarrow \pi_n(\mathbb{CP}^2, b_0)$ is an isomorphism for every $n\geq 1$. Therefore, in order to prove that $\Phi_0$ is non-null-homotopic, it suffices to lift it to a map from the $2$-disk $\mathbb{D}$ with target space $S^5$, restrict this map to the boundary, and check that final object is non-null-homotopic. For simplicity, let $\Phi_0(x) = \{x, \infty\}$, for every $x \in S^2$, where the fixed point $x_0 = \infty$ is the point at infinity. Then, $\Phi_0(x) = [(z+x)] = [(x, 1, 0)]$, if $x\neq \infty$, and  $\Phi_0(\infty) = [(1, 0, 0)]$, where the equivalence class of the polynomial $[a_0 + a_1 z + a_2 z^2]$ is associated to $[(a_0, a_1, a_2)] \in \mathbb{CP}^2$. This map can be presented as a map from $\mathbb{D}$ to $\mathbb{C}\mathbb{P}^2$ that is constant equal to $[(1,0,0)]$ on the boundary circle. We still use $\Phi_0$ to denote the different ways to represent this map, i.e., we use
$$
\Phi_0 : \mathbb{D} \rightarrow \mathbb{CP}^2, \quad \Phi_0(w):= \left[\left( \frac{w}{\sqrt{1-|w|^2}}, 1, 0 \right)\right], 
$$
for $|w|<1$, and $\Phi_0(w) = [(w,0,0)] = [(1,0,0)]$, for $|w|=1$. Notice that this map can be viewed as a representative of the homotopy class of the original $\Phi_0$ because $w \mapsto w(1-|w|^2)^{-\frac{1}{2}}$ is a homeomorphism from $\{w \in \mathbb{C} : |w|<1\}$ to $\mathbb{C}$, and $\Phi_0(w) = [(w, \sqrt{1-|w|^2}, 0)]$, for $|w|<1$, allows us to check that the new map continuously extends to $\mathbb{D}$.

Moreover, the quotient space of $\mathbb{D}/\sim$ under the equivalence relation that collapses the boundary to a point is homeomorphic to the Riemann sphere via the same map $w\mapsto w(1-|w|^2)^{-\frac{1}{2}}$, for $|w|<1$, and $\partial \mathbb{D} \mapsto \infty$. Under this homeomorphism, the two maps $\Phi_0$ that we defined are exactly the same.

In addition, writing $\Phi_0(w) = [(w, \sqrt{1-|w|^2}, 0)]$ also allows us to lift $\Phi_0$ as the map $\psi : (\mathbb{D}, \partial \mathbb{D}, y_0) \rightarrow (S^5, S^1, q_0)$, $\psi (w):= (w, \sqrt{1-|w|^2}, 0)$. Indeed, $p(\psi(w)) = [(w, \sqrt{1-|w|^2}, 0)] = \Phi_0(w)$, for every $w \in \mathbb{D}$. Here we use the choice of base points $b_0 = [(1,0,0)]$, any $y_0 \in \partial
\mathbb{D}$, $q_0 = (y_0, 0, 0) = \psi(y_0)$, and $S^1 = \{(y, 0, 0) : y\in \mathbb{C}, |y|=1\}$.
Therefore, the restriction of $\psi$ to the boundary is $\psi|_{\partial\mathbb{D}} : (\partial\mathbb{D}, y_0) \rightarrow (S^1, q_0), \,\, y\mapsto (y,0,0)$. Since this map is non-null-homotopic, the map $\Phi_0$ shares the same property.
 
\subsection{Maps whose image contains trivial pairs only are not sweepouts}\label{subsect-not-sweepout}

For simplicity, let us drop the embedding from $S^2$ to a Riemannian manifold and use $\mathcal{P}_{S^2}$ and the map $\Phi_0(x):=\{x, \infty\}$, for every $x\in S^2$. The goal of this section is to prove that $\Phi_0$ is not homotopic to any map of the form $\Phi_1 : S^2 \rightarrow \mathcal{P}_{S^2}$ such that $\Phi_1(x) = \{\phi_1(x)\}$ for some continuous map $\phi_1 : S^2 \rightarrow S^2$. Using a rotation, $\Phi_1$ can be homotoped to $\Phi_2(x) = \{\phi_2(x)\}$, for which $\phi_2(\infty) = \infty$. In order to show that $\Phi_2$ and $\Phi_0$ are not homotopic, we prove that the induced maps on the second homotopy group $(\Phi_0)_{\ast}, (\Phi_2)_{\ast} : \pi_2(S^2, \infty) \rightarrow \pi_2 (\mathcal{P}_{S^2}, \{\infty\})$ are different. The general case of embedded spheres follows similarly.

In order to achieve our goal, let us recall the following general facts. We use $I^n = [0,1]^n$ to denote the $n$-cube.

\begin{lemm}\label{lemm-relationship-induced-maps}
Let $X$ and $Y$ be topological spaces, and $\varphi_t: X \rightarrow Y$ be a homotopy, $0\leq t \leq 1$. Fixed $x_0 \in X$, let $h$ be the curve $h(t):= \varphi_t(x_0)$, $t\in [0,1] $. Let $\xi_h : \pi_n(Y, \varphi_1(x_0)) \rightarrow \pi_n(Y, \varphi_0(x_0))$ be the map defined as follows: the homotopy class of a given map $f : (I^n, \partial I^n) \rightarrow (Y, \varphi_1(x_0))$, is mapped by $\xi_h$ to the homotopy class of $f_h : (I^n, \partial I^n) \rightarrow (Y, \varphi_0(x_0))$ that is obtained by shrinking the domain of $f$ to a smaller concentric $n$-cube, and inserting copies of the path $h$ from $\varphi_0(x_0)$ to $\varphi_1(x_0)$ on each radial segment in the region between $\partial I^n$ and the boundary of the smaller cube. 

Then, the induced maps $(\varphi_t)_{\ast} : \pi_n (X, x_0) \rightarrow \pi_n(Y, \varphi_i(x_0))$ are related by $(\varphi_0)_{\ast} = \xi_h \circ (\varphi_1)_{\ast}$. Moreover, $\xi_h$ depends on the homotopy class of $h$ only. 
\end{lemm}

The case $n=1$ is proved in Lemma 1.19 of \cite{Hat}, and the general case is completely analogous.

In our setting, $\Phi_0(\infty)=\Phi_2(\infty)$ and the target space $\mathcal{P}_{S^2}$ is simply connected. Therefore, assuming, by contradiction, that $\Phi_0$ is homotopic to $\Phi_2$, the path $h$ traced by the image of $\infty \in S^2$ along the homotopy would be a contractible loop. In particular, Lemma \ref{lemm-relationship-induced-maps} would yield $(\Phi_0)_{\ast}=(\Phi_2)_{\ast}$.

In Sections \ref{subsect-topology-space-pairs} and \ref{subsect-std-sweep} we obtained that $\mathcal{P}_{S^2}$ is homeomorphic to $\mathbb{CP}^2$ and that $(\Phi_0)_{\ast}$ takes a generator of $\pi_2(S^2, \infty)$ to the generator of $\pi_2(\mathcal{P}_{S^2}, \{\infty\})$, which is isomorphic to $\pi_1(S^1) = \mathbb{Z}$. We claim that $(\Phi_2)_{\ast}$ takes the generator of $\pi_2(S^2, \infty)$ to an even homotopy class of $\pi_2(\mathcal{P}_{S^2}, \{\infty\})$. It is convenient to write $\Phi_2 = \mathcal{I} \circ \phi_2$, where $\mathcal{I} : S^2 \rightarrow \mathcal{P}_{S^2}$ is defined by $\mathcal{I}(x) = \{x\}$, for all $x\in S^2$. Then, $(\Phi_2)_{\ast} = \mathcal{I}_{\ast} \circ (\phi_2)_{\ast}$. Hurewicz's Theorem implies that $\pi_2(S^2) = \mathbb{Z}$. Therefore, it suffices to show that $(\mathcal{I})_{\ast}$ takes the generator of $\pi_2(S^2, \infty)$ to the class corresponding to the integer $2$ in $\pi_2(\mathcal{P}_{S^2}, \{\infty\}) = \mathbb{Z}$. The computation is similar to that made in Section \ref{subsect-std-sweep}.

Using the identification of $\mathcal{P}_{S^2}= \mathbb{CP}^2$ and the fiber bundle structure $S^1\rightarrow S^5 \rightarrow \mathbb{CP}^2$, consider the map $f : \mathbb{D} \rightarrow \mathcal{P}_{S^2}$, defined by
$$
f(w) := \left[\left( z + \frac{w}{\sqrt{\sqrt{1+3|w|^4}-2|w|^2}}\right)^2\right],
$$
for $|w|<1$, and $f(w) = [(1,0,0)]$, for $w\in \partial \mathbb{D}$. The homotopy class of this map represents the image of the generator of $\pi_2(S^2, \infty)$ under $(\mathcal{I})_{\ast}$. Consider $\tilde{f}: (\mathbb{D}, \partial\mathbb{D}, y_0) \rightarrow (S^5, S^1, q_0)$ given by
$$
\tilde{f}(w) = \left(w^2, 2w\sqrt{\sqrt{1+3|w|^4}-2|w|^2}, \sqrt{1+3|w|^4}-2|w|^2\right),
$$
for every $w \in \mathbb{D}$, where $y_0 \in \partial\mathbb{D}$ and $q_0 = (y_0^2, 0,0)$. Notice that $\tilde{f}$ is a lift of $f$ to the total space of the fiber bundle, and its restriction to the boundary is given by $\tilde{f}|_{\partial\mathbb{D}} : (\partial\mathbb{D}, y_0) \rightarrow (S^1, q_0)$, $y\mapsto (y^2, 0,0)$. As $y$ moves along the circle $\partial \mathbb{D}$ once, $(y^2, 0, 0)$ moves around $S^1$ twice. Then, the homotopy class of the restriction of $\tilde{f}$ to the boundary corresponds to the integer $2$. 


\subsection{Regular and critical points}\label{subsect-regular-and-critical}

Let $\Sigma$ be a smoothly embedded $2$-sphere in a complete Riemannian manifold $(M^n,g)$. The goal of this section is to provide the proofs of the the basic claims involving critical points of the distance function explained in Section \ref{subsect-width-introd}, which includes the proof of the first Min-max Theorem \ref{thm-W_d-critical}. Recall the terminology in that section.

\subsubsection{Characterization of critical points}\label{subsect-characterization-critical-vs-simult-st}
We begin with the characterization of critical points in terms of simultaneously stationary geodesics.

\begin{proof}[Proof of Proposition \ref{prop-characterization-critical-simult-st}]
Assume that $\{\gamma_1, \ldots, \gamma_k\}$ is a set of simultaneously stationary geodesics with endpoints $x$ and $y \in \Sigma$. Definition \ref{defi-simultaneously-stationary} implies that $\sum_{i=1}^k c_i \nu_{\gamma_i}^T = 0$, for a choice of constants $0<c_i\leq 1$, $i=1, \ldots, k$, whose sum is equal to one. Suppose, by contradiction, that $\{x, y\}$ is a regular point. In particular, there exists $w_1 \in T_x\Sigma$ and $w_2 \in T_y\Sigma$ satisfying $g((w_1, w_2), \nu_{\gamma_i}) < 0$, for all $i$. Since the $w_j$ are tangential to $\Sigma$ and $c_i$ are positive, it follows that $g\left( (w_1, w_2), \sum_{i=1}^k c_i \nu_{\gamma_i}^T\right) <0$. Which is a contradiction.

In order to prove the second claim of the statement, we apply the following theorem of Carath\'eodory. Let $x$ be a point in the convex hull of a subset $T\subset\mathbb{R}^n$ ($x$ is a convex combination of a finite number of elements of $T$). Then, $x$ is contained in the convex hull of a set $T'\subset T$ such that $\#T^{\prime} \leq n+1$. See Theorem 2.8.2 of \cite{MarMonOli}, for instance.

Another fact that we use is that if $A\subset \mathbb{R}^n$ is compact, then the following are equivalent: (i) there exists $v \in \mathbb{R}^n$ such that $\langle v, a \rangle < 0$, for every $a\in A$, and (ii) the origin does not belongs to the convex hull of $A$, i.e., $0\notin conv(A)$. The convex hull of $A$ is the set made of all finite linear combinations of points in $A$. Let us indicate the proof of that fact. An easy consequence of Carath\'eodory's theorem is that the compactness of $A$ implies that $conv(A)$ is also compact. Then, (ii) is equivalent to the fact that $0$ and $conv(A)$ can be strictly separated by a straight line, and the equivalence with (i) follows.

Next, we prove the second statement of Proposition \ref{prop-characterization-critical-simult-st}. Let $\{e_1, e_2\}\subset T_x\Sigma$ and $\{f_1, f_2\}\subset T_y\Sigma$ be orthonormal bases. Consider $\mathbb{R}^4$ equipped with the Euclidean inner product, and $T_x\Sigma\times T_y\Sigma$ with the inner product $g$ as in Definition \ref{defi-regular-point}. The map $(a_1, a_2, b_1, b_2) \in\mathbb{R}^4 \mapsto (a_1e_1 + a_2e_2, b_1f_1+b_2f_2) \in T_x\Sigma\times T_y\Sigma$ is a linear isometry. Therefore, we can apply the fact stated in the preceding paragraph on $T_x\Sigma \times T_y\Sigma$ to conclude that $\{x,y\}$ is regular if, and only if, $(0,0) \notin conv(A)$, where $A$ is the compact set made of all  $\nu_{\gamma}^T = (-\gamma'(0)^T, \gamma'(a)^T)$, where $\gamma$ are minimizing geodesics joining $x = \gamma(0)$ and $y = \gamma(a)$. In particular, if $\{x,y\}$ is critical, then $(0,0) \in conv(A)$, and, by Carath\'eodory's theorem, the origin is a linear combination of a set of at most five vectors of the form $\nu_{\gamma}^T$.
\end{proof}

\subsubsection{A gradient-type vector field and the proof of Theorem \ref{thm-W_d-critical}} 

In this section we summarize the Morse theoretic arguments necessary to the proof of Theorem \ref{thm-W_d-critical}. We begin by collecting some key arguments involving useful properties of a gradient-type vector field in the following proposition. The following result is similar to Proposition 2.5 of \cite{AMS}.

\begin{prop}\label{prop-gradient-like}
Let $\Sigma$ be a smoothly embedded $2$-sphere in a complete Riemannian manifold $(M^n,g)$, and $\mathcal{R}\subset \mathcal{P}_{\Sigma}$ be the set of regular points $\{x, y\}$ of the distance function. The set $\mathcal{R}$ is an open subset, and there exists a vector field $Y$ on $\mathcal{R}$ whose flow $\{\psi_t\}$ satisfies the following property: for every compact set $K \subset \mathcal{R}$, there exists $c=c(K)>0$ such that
\begin{align} \label{eqtn-prop=gradient-like}
\frac{dist(\psi_t(\{x, y\})) - dist(x,y)}{t} \leq -c,  
\end{align}
for $\{x, y\} \in K$ and $t\neq 0$, possibly negative, while $\psi_{\tau}(\{x, y\}) \in K$, for every $\tau$ between $0$ and $t$, not including $\tau=0$. In particular, any local maximum or local minimum of the distance is a critical point.
\end{prop}

\begin{proof}
The first step is the construction of a local vector field $\mathcal{Y}(x_0,y_0)$ around a regular point $\{x_0, y_0\} \in \mathcal{P}_{\Sigma}$ of the distance function for which $g\left( \mathcal{Y}(x_0, y_0)|_{\partial\gamma}, \nu_{\gamma}^T\right)< 0$, for every minimizing geodesic $\gamma$ with endpoints $\partial \gamma \in \mathcal{P}_{\Sigma}$ close to $\{x_0, y_0\}$. This shows, in particular, that $\mathcal{R}$ is open. It suffices to extend smoothly and tangentially to a neighborhood of $(x_0, y_0)$ a choice of a pair of vectors in $T_{x_0}\Sigma \times T_{y_0}\Sigma$ satisfying the desired property. The claim follows from the compactness of the set of minimizing geodesics of $(M^n, g)$ with endpoints in a fixed compact subset. In particular, a sequence of minimizing geodesics parametrized proportionally to arc length on $[0,1]$, whose endpoints converge to $x_0$ and $y_0$, smoothly subconverges to a minimizing geodesic joining $x_0$ and $y_0$. Since $x_0\neq y_0$, there is no ambiguity in the definition of the vector field.

Apply a partition of unity subordinated to a locally finite covering of $\mathcal{R}$ by the domains of the local fields considered above, and normalize to obtain a unit vector field $Y$. Since $Y(x,y)$ is a linear combination with positive coefficients of finitely many vectors of the form $\mathcal{Y}(x_i, y_i)$, it follows that $g(Y(x,y), \nu^T_{\gamma}) < 0$, for every $(x,y) \in \mathcal{R}$ and minimizing geodesic $\gamma$ with endpoints $x$ and $y$.

The estimates (\ref{eqtn-prop=gradient-like}) are obtained by comparing the evolution of the distances $dist(\psi_t(\{x,y\}))$ with the evolution of the length of any minimizing geodesic joining $x$ and $y$ along a smooth family of smooth curves with endpoints $\psi_t(\{x, y\})$. Indeed, let $K^{\prime} \subset \mathcal{R}$ be a compact set that contains $K$ in its interior. The flow $\psi_t(\{x,y\})$ of $Y$, starting at $\{x,y\} \in K$, exists and stays in $K^{\prime}$ for uniformly positive and negative times. There exists a positive constant $c$ satisfying $g(Y(x,y), \nu^T_{\gamma})\leq -c$, for every $(x,y) \in K^{\prime}$ and minimizing geodesic $\gamma$ with endpoints $x$ and $y$. Since the two points in regular pairs are always different, one can write $\psi_t(\{x,y\}) = \{x_t, y_t\}$, where $x_t, y_t \in \Sigma$ are smooth functions of the parameter $t$. For every $\{x_{t_0}, y_{t_0} \} \in K^{\prime}$, let $\{\alpha_t\}$ be a smooth family of curves with endpoints $\partial \alpha_t = \{x_t, y_t\}$, for $t$ close to $t_0$, and such that $\alpha_{t_0}$ is a minimizing geodesic joining $\{x_{t_0}, y_{t_0} \}$. Hence, $dist(\{x_{t}, y_{t} \}) \leq L(\alpha_t)$, for all $t$ close to $t_0$, with equality holding for $t=t_0$. The first variation formula implies that the derivative of $dist(\{x_{t}, y_{t} \})$ is bounded from above by $-c$ in the sense of support functions. Completing the proof of the proposition.
\end{proof}

We are now ready to prove the first min-max theorem.

\begin{proof}[Proof of Theorem \ref{thm-W_d-critical}]
Let us prove that $W_d(\Sigma)>0$. Consider the set 
$$
\mathcal{A}:= \{ \{i(x), i(-x)\} : x \in S^2\} \subset \mathcal{P}_{\Sigma},
$$
where $i:S^2\rightarrow M$ is the fixed embedding of $S^2$ whose image is $\Sigma$.

\textbf{Claim:} The image of every sweepout $\phi:S^2\rightarrow \mathcal{P}_{\Sigma}$ of $\Sigma$ intersects $\mathcal{A}$.

We prove this claim by contradiction. Suppose that there exists a sweepout $\phi$ whose image does not intersect $\mathcal{A}$. Then, for each $x\in S^2$, we can write $\phi(x) = \{i(P), i(Q)\}$, for some $P, Q \in S^2$ with $Q\neq -P$. Let $[P,Q]$ denote the unique (nonoriented) shortest arc of great circle in $S^2$ joining $P$ and $Q$, and let $M_x$ be the midpoint of this arc. Shrinking $[P,Q]$ to $M_x$ allows us to continuously deform the pair $\{i(P), i(Q)\}=\phi(x)$ to the trivial pair $\{i(M_x)\}$. This deformation can be done systematically for all $x$, showing that $\phi$ is homotopic to the continuous map $\Phi_1 : S^2 \rightarrow \mathcal{P}_{\Sigma}$ given by $\Phi_1(x) = \{i(M_x)\}$. We proved, in Section \ref{subsect-not-sweepout}, that $\Phi_1$ is not a sweepout, which contradicts the fact that $\phi$ is a sweepout. The claim is proved and yields, by a direct compactness argument on $S^2$,
$$
W_d(\Sigma) \geq \inf \{dist(i(x), i(-x)) : x \in S^2\} >0.
$$

It remains to show that $W_d(\Sigma)$ is a critical value of the distance function. This is also obtained by contradiction, and the argument is analogous to the proof of Theorem 3.1 of \cite{AMS}. For the convenience of the reader we outline the proof here.  Suppose that there exists $\varepsilon>0$ such that the compact subset 
	\begin{align*}
		K = \{\{x,y\}\in \mathcal{P}_{\Sigma} : W_d(\Sigma)-\varepsilon\leq dist(x,y)\leq W_d(\Sigma)+\varepsilon\}
	\end{align*}
does not contain critical points of the distance. Since $W_d(\Sigma)>0$, we can assume, without loss of generality, that $W_d(\Sigma)-2\varepsilon >0$.

Consider the vector field $\zeta Y$, where $Y$ is the vector field constructed in Proposition \ref{prop-gradient-like}, and $\zeta$ is a smooth cut-off function satisfying $0\leq \zeta\leq 1$, $\zeta|_{K} = 1$, and supported on a sufficiently small open neighborhood of $K$. Assume further that the support of $\zeta$ is still contained in the set $\mathcal{R}$ of regular points, and does not contain points such that $dist(x,y) < W_d(\Sigma)-2\varepsilon$. Extend $\zeta Y$ by zero outside the region $\mathcal{R}$ of regular points. 

We claim that the flow of $\zeta Y$ on $\mathcal{P}_{\Sigma}$ retracts the set of points on which the distance function is at most $W_d(\Sigma)+\varepsilon$ into a subset of $\mathcal{P}_{\Sigma}$ contained in $\{ \{x,y\} \in \mathcal{P}_{\Sigma} : dist(x,y) < W_d(\Sigma)-\varepsilon\}$, in such way that the points $\{x,y\}$ with $dist(x,y) < W_d(\Sigma)-2\varepsilon$ do not move. Indeed, $\zeta Y$ coincides with $Y$ on the subset $K$, and since the flow $\psi_t$ of $Y$ uniformly decreases distance (by Proposition \ref{prop-gradient-like}), for each $\{x,y\} \in K$, the path $\psi_t(\{x,y\})$, $t\geq 0$, only leaves the region $K$ after $dist(\psi_t(\{x,y\}))$ achieves the value $W_d(\Sigma)-\varepsilon$.

Let $\phi: S^2 \rightarrow \mathcal{P}_{\Sigma}$ be a sweepout such that $\max_{x \in S^2} dist(\phi(x)) < W_d(\Sigma)+\varepsilon$.
Deform $\phi$ by the flow of $\zeta Y$ to obtain another sweepout $\tilde{\phi}$ satisfying
\begin{equation}\label{eq-proof-theo-minimax}
\max_{x \in S^2} dist(\tilde{\phi}(x)) < W_d(\Sigma)-\varepsilon.
\end{equation}
The fact that $\tilde{\phi}$ is a sweepout can be justified as follows. The support of $\zeta Y$ is away from the set of pairs of points where the distance function vanishes. In particular, the deformation of a pair $\phi(x) = \{\phi_1(x), \phi_2(x)\}$ by the flow of $\zeta Y$ creates a curve of pairs of distinct points, unless $\phi(x)$ is fixed along the flow of $\zeta Y$. Those form two paths in $\Sigma$, starting at $\phi_i(x)$ and ending at $\tilde{\phi}_i(x)$, for $i=1,2$. This implies that $\tilde{\phi}$ is also of the form $\tilde{\phi}=\{\tilde{\phi}_1, \tilde{\phi}_2\}$, for a pair of continuous functions $\tilde{\phi}_i : S^2 \rightarrow \Sigma$. Therefore, $\tilde{\phi}$ is an admissible sweepout, and inequality (\ref{eq-proof-theo-minimax}) contradicts the definition of $W_d(\Sigma)$.
\end{proof}

We finish this section with an example of a rotationally symmetric, positively curved metric on $S^3$, which contains an embedded $2$-sphere $\Sigma$ whose width $W_d(\Sigma)$ is realized by critical points $\{x, y\}$ in such a way that the minimizing geodesics joining $x$ and $y$ are not free boundary on $\Sigma$.

\subsection{Example: minimizing geodesics whose endpoints realize $W_d(\Sigma)$ do not satisfy the free-boundary condition always}\label{exam-example-ellipsoid}
For $a>1$, consider the ellipsoid of revolution
	\begin{equation*}
		\mathcal{E}_a^3=\Big\{(x,y,z,w)\in \mathbb{R}^4\,|\, x^2+y^2+z^2+\frac{w^2}{a^2} = 1 \Big\},
	\end{equation*}
equipped with the Riemannian metric inherited from the standard flat metric of the Euclidean space $\mathbb{R}^4$. Let $\mathcal{E}_a^2 = \mathcal{E}_a^3 \cap \{z=0\}$. Given $\delta>0$, let us use the following terminology: $\Sigma(a, \delta) = \mathcal{E}_{a}^3\cap\{w = a\delta\} $, $\Gamma(a,\delta)=\mathcal{E}_{a}^2\cap\{w = a\delta\}$, and $A=(\sqrt{1-\delta^2},0, 0, a\delta), B = (-\sqrt{1-\delta^2},0,0,a\delta) \in \Gamma(a,\delta)$. We claim that the extrinsic min-max width $W_d(\Sigma(a, \delta))$, associated to the distance function of $\mathcal{E}_a^3$, is realized by the length of minimizing geodesics of  $\mathcal{E}_{a}^3$ joining $A$ and $B$, which are not orthogonal to $\Sigma(a, \delta)$, if $\delta>0$ is sufficiently small. Moreover, there exists an infinite number of such minimizing geodesic.

Before proving the claims made in the preceding paragraph, let us look first at $\mathcal{E}_a^2$. We claim that there are two distinct minimizing geodesics of $\mathcal{E}_a^2$ joining $A$ and $B$, both contained in the region where $w\geq a\delta$, and not orthogonal to $\Gamma(a, \delta)$. We prove first that the radial projection of $\mathcal{E}_a^2$ over the unit sphere $\{(x, y, 0, w) \in \mathbb{R}^4\,|\, x^2+y^2+w^2 =1\}$, the map $(x,y,0,w)\mapsto (x^2+y^2+w^2)^{-1/2}(x,y,0,w)$, is 1-Lipschitz. Indeed, its derivative at $p$ in the direction of the vector $v$ is given by $v/|p| - \langle p, v\rangle|p|^{-3} p$, which has norm bounded by $|v|$, whenever $|p|\geq 1$. Moreover, equality holds if, and only if, $|p|=1$ and $\langle p, v\rangle =0$. In particular, any portion of the equator $\mathcal{E}_a^2\cap \{w=0\}$ of length $\pi$ is minimizing, and these are the only minimizing geodesics of $\mathcal{E}_a^2$ joining those endpoints.

Next, we prove that every geodesic $\alpha(t)=(x(t), y(t), 0, w(t))$ of $\mathcal{E}_a^2$, with $\alpha(0) = A$ and $w'(0)\leq 0$ must cross the equator, $\mathcal{E}_a^2\cap \{w=0\}$, before possibly passing through the point $B$. Indeed, $w'(0)\leq 0$ implies that, for small positive times, the geodesic would be below the horizontal slice $\{w=a\delta\}\cap \mathcal{E}_a^2$. Hence, its distance to the axis of revolution would be increasing and, by Clairaut's relation, its angle with the meridians would also increase. In conclusion, if $\alpha$ joins $A$ and $B$, it must cross the equator twice and can not be minimizing from $A$ to $B$. This proves that minimizing geodesics of $\mathcal{E}_a^2$ joining $A$ and $B$ are contained in the region where $w\geq a\delta$.

It remains to prove that the minimizing geodesics are not orthogonal to $\Gamma(a, \delta)$, for small $\delta$. Suppose, by contradiction, that there exists a sequence $\delta_i \rightarrow 0$ for which the geodesics of $\mathcal{E}_a^2$, joining $A_i$ and $B_i$, and orthogonal to $\Gamma(a, \delta_i)$ are minimizing, for all $i$. By compactness, these curves smoothly converge to a minimizing geodesic joining $(1,0,0,0)$ and $(-1,0,0,0)$, and is orthogonal to the equator $\Gamma(a, 0)$. Which gives us a contradiction, and completes the argument. Finally, notice that the reflection symmetry across $y=0$ shows that there must be two minimizing geodesics joining $A$ and $B$.

Let us approach the case of $\mathcal{E}_a^3$. Reflection across the hyperplane $z=0$ leaves the three-dimensional ellipsoid $\mathcal{E}_a^3$ invariant and fixes $\mathcal{E}_a^2$. Therefore, $\mathcal{E}_a^2$ is totally geodesic, and curves $\alpha(t)=(x(t), y(t), 0, w(t))$ with $\alpha(0) = A$ that are geodesics of $\mathcal{E}_a^2$, are also geodesics of $\mathcal{E}_a^3$. The other geodesics emanating from $A$ are obtained by rotation and have the form $\alpha_{\theta}(t) := (x(t), \cos \theta \cdot y(t), \sin \theta\cdot  y(t), w(t))$, for some $\theta \in [0, 2\pi]$. By the argument explained before, if $\alpha_{\theta}$ is minimizing from $A$ to $B$, then $\alpha$ is also minimizing joining the same endpoints. In particular, $\alpha$ is minimizing on $\mathcal{E}_a^2$. In this case, the result of the preceding paragraphs implies that $w(t)\geq a\delta$, and, by rotational symmetry, there exist infinitely many such geodesics. Moreover, for small $\delta$, the minimizing curves are not orthogonal to $\Sigma(a, \delta)$.

It remains to prove that $W_d(\Sigma(a, \delta)) = dist(A,B)$. Applying the result of the claim contained in the proof of Theorem \ref{thm-W_d-critical} for the embedding of the standard sphere to $\Sigma(a, \delta)$ defined by  $(x,y,z)  \mapsto \sqrt{1-\delta^2}(x,y,z,\frac{a\delta}{\sqrt{1-\delta^2}})$ yields the following lower bound for the min-max width:
$$
W_d(\Sigma(a, \delta)) \geq \inf\{ dist((x,y,z,a\delta), (-x,-y,-z,a\delta)) : x^2+y^2+z^2=1-\delta^2\}.
$$
Rotation on $S^2$ implies that $dist((x,y,z,a\delta), (-x,-y,-z,a\delta))$ is constant equal to $ dist(A,B)$. In order to provide an upper bound, we prove that $dist(A,B)$ equals the extrinsic diameter of $\Sigma(a, \delta)$ in $\mathcal{E}_a^3$. Assuming this fact, a standard sweepout based at one point would yield the upper bound. 

It remains to prove the monotonicity of the distance on $\mathcal{E}_a^3$ between the points $A_1 = (x_1, y_1, z_1, a\delta)$ and $A_2 = (x_2, y_2, z_2, a\delta)$ with respect to the angular separation of $(x_1, y_1, z_1)$ and $(x_2, y_2, z_2)$. Any ambient smooth curve joining $A_1$ and $A_2$ can be written in the form $(\cos \theta(t) u(t), a \sin \theta(t))$, $t\in [0,1]$, for some functions $\theta = \theta(t) \in \mathbb{R}$ and $u=u(t) \in S^2$. Given $t_0$, consider the piecewise smooth curve parametrized by $v(t) = (\cos \theta(t) u(t), a \sin \theta(t))$, for $t\in [0,t_0]$, and $v(t) = (\cos \theta(t) u(t_0), a \sin \theta(t))$, for $t\in [t_0, 1]$. Thus, $v$ is contained in $\mathcal{E}_a^3$, has endpoints in $\Sigma(a, \delta)$, its length is bounded from above by the length of the initial curve, and the angular separation is now that between the unit vectors $u(0)$ and $u(t_0)$. The fact that any angular separation from zero to the separation between $(x_1, y_1, z_1)$ and $(x_2, y_2, z_2)$ is the angular separation between vectors of the form $u(0)$ and $u(t_0)$, for some $t_0 \in [0,1]$, implies the desired monotonicity.

\section{Width of 2-spheres whose closed geodesics are all unstable}\label{sect-intrinsic-thmb}

In this section we prove the results stated in Section \ref{subsect-intrinsic-thmb}. We also provide details of the relevant examples related to those theorems.

\subsection{Proof of Theorem \ref{thm-3-simult-st}}\label{subsect-proof-thm-3-simult-st}


In this section, we prove Theorem \ref{thm-3-simult-st}. One key part of the argument is to show that if $\{x,y\}$ is a critical point of the distance function realizing $W_d(\Sigma)$, and $\gamma_1, \gamma_2, \gamma_3 : [0, a] \rightarrow S^2$ are distinct minimizing geodesics joining $\gamma_i(0)=x$ and $\gamma_i(a)=y$ satisfying that the three angles determined by the vectors $\gamma_i^{\prime}(0)$ at the origin of $T_xS^2$, and the three angles determined by the vectors $-\gamma_i^{\prime}(a)$ at the origin of $T_yS^2$ all measure at most $\pi$, then $\{\gamma_1, \gamma_2, \gamma_3\}$ is simultaneously stationary with endpoints $x$ and $y$. If the critical point $\{x,y\}$ realizing $W_d(\Sigma)$ bounds two minimizing geodesics $\gamma_{1}$ and $\gamma_{2}$ whose union is a closed geodesic $\gamma_{1}\cup\gamma_{2}=\gamma$, then $\gamma$ has index one. Before proving Theorem \ref{thm-3-simult-st}, we present a few results that are important tools in the construction of sweepouts. The first is the following standard lemma whose proof is omitted here. 

\begin{lemm}\label{funct-exist-lem}
Let $\gamma$ be a closed geodesic and let $V$ be a vector field along $\gamma$ such that $\{\gamma^{\prime}(s), V(\gamma(s))\}$ is an orthonormal basis along $\gamma$. If $\gamma$ is unstable, then there exists a positive function $u$ on $\gamma$ such that $\mathcal{Q}(u V,u V)<0$.
\end{lemm}

In the case of three distinct minimizing geodesics, consider the following terminology. Let $\theta_1\in(0,\pi]$ be the measure of the angle determined by the vectors $\gamma^{\prime}_1(0)$ and $\gamma^{\prime}_2(0)$ in the region of $T_xS^2$ bounded by the half-lines determined by those vectors that does not contain $\gamma^{\prime}_3(0)$. Similarly, define the angles $\theta_2$ between 
$\gamma^{\prime}_2(0)$ and $\gamma^{\prime}_3(0)$, $\theta_3$ between $\gamma^{\prime}_3(0)$ and $\gamma^{\prime}_1(0)$, and $\theta_i^{\ast}$ the analog angles determined by the vectors $-\gamma_1^{\prime}(a)$, $-\gamma_2^{\prime}(a)$, and $-\gamma_3^{\prime}(a)$.

\begin{prop}\label{prop-exis-swee}
Let $(S^{2},g)$ be a Riemannian sphere whose closed geodesics are all unstable. Let $\{x,y\}\subset S^{2}$ be a critical point of the distance function with $x\neq y$. Suppose that $x$ and $y$ are joined by three distinct minimizing geodesics $\gamma_{i}$, $i\in\{1,2,3\}$, such that $\gamma_1\cup\gamma_2\cup\gamma_3$ does not contain closed geodesics. In addition, assume that the angles $\theta_i$ and $\theta_i^{\ast}$ described above measure at most $\pi$, for $i\in\{1,2,3\}$. Then, there exists a sweepout $\psi\in[\Phi_0]$ of $S^2$ such that $dist(\psi(p))\leq dist(x,y)$, for every $p\in S^2$. Moreover, $dist(\psi(p)) = dist(x,y)$ if, and only if, $\psi(p)=\{x,y\}$.
\end{prop}

\begin{proof}		
We write $\Omega_{1}$, $\Omega_{2}$, and $\Omega_{3}$ to represent the closure of the connected components of $S^2\setminus\{\gamma_{1}\cup\gamma_{2}\cup\gamma_{3}\}$ bounded by $\gamma_{1}\cup\gamma_{2}$, $\gamma_{2}\cup\gamma_{3}$, and $\gamma_{3}\cup\gamma_{1}$, respectively. Initially, let $\gamma_{l}:[0,a]\to\Omega_{j}$ be parametrized by arc-length, $\gamma_{l}(0)=x$ and $\gamma_{l}(a)=y$, and use $\gamma_{l}\cup\gamma_{k}:[0,a]\to\Omega_{j}$ as the parametrization 
of the union of those two curves given by $\gamma_{l}\cup\gamma_{k}(t)=\gamma_{l}(2t)$, for $t\in\left[ 0,\frac{a}{2}\right]$, and $\gamma_{l}\cup\gamma_{k}(t)=\gamma_{k}(2a-2t)$, for $t\in\left[ \frac{a}{2}, a \right]$. By the angle condition on the curves $\gamma_{l}$ and $\gamma_{k}$, the regions $\Omega_j$, with $\partial\Omega_{j}=\gamma_{l}\cup\gamma_{k}$, are locally convex in the sense that minimizing geodesics joining sufficiently close points in $\Omega_{j}$ are contained in that region. 

The next step is to completely sweep the three regions out by curves. This will be achieved by iterating the Birkhoff curve shortening process. This process involves shortening piecewise smooth closed curves of length at most $L$. Let $N>2$ be a positive integer for which: every piecewise smooth closed curve $c$ of length at most $L$, parametrized proportionally to arc-length on the interval $[0, a]$ (or any fixed interval independent of $c$), is such that
the partition $0=t_0<\cdots< t_N=a$ of the interval $[0,a]$ into $N$ intervals of the same size satisfies that the image $c|_{[t_{i-1},t_i]}$ has length smaller than the injectivity radius of $(S^2,g)$. Formally, one must use curves of uniformly bounded energy. Briefly, the process is divided into two parts. The first part involves a geodesic approximation of $c$ by a broken geodesic determined by the points of the curve associated to the fixed partition of $[0,a]$. In the second part, a second geodesic approximation is performed using the midpoints of each minimizing geodesic segment from the first step. The process allows for an interpolation between the initial and the final curves which consists in replacing small pieces of the curve by minimizing geodesics joining their endpoints. Therefore, throughout the process the length is decreasing. We refer the reader to \cite{AMS}, \cite{DM}, and Section 2 of \cite{CC} for detailed descriptions.

The lengths of the curves through an iteration of the process either decrease to zero, or converge to a positive number. In the second case, the iterated curves converge to a closed geodesic. See, for instance, Section 2 of \cite{CC}. By the local convexity of $\Omega_j$, the iterated Birkhoff's process starting at the curve $\gamma_{l}\cup\gamma_{k}$ will remain in the region $\Omega_{j}$ bounded that curve.

\textbf{Claim 1}: Iteration of the Birkhoff curve shortening process starting at the curves $\gamma_{l}\cup\gamma_{k}$ end at points, in the sense that the length of the iterated curves decrease to zero.

\begin{proof}[Proof of the Claim 1:] 
Suppose, by contradiction, that the statement is false. Hence, the process starting at $\gamma_{l}\cup\gamma_{k}$ ends at an immersed geodesic $\sigma$ of length $L(\sigma)< L(\gamma_{l}\cup\gamma_{k})$ contained in the interior of $\Omega_j$. Furthermore, $\sigma$ does not intersect $\gamma_{l}\cup\gamma_{k}$; otherwise, $\sigma$ and $\gamma_{l}\cup\gamma_{k}$ would share the same tangent vector at some point, implying $\gamma_{l}\cup\gamma_{k}=\sigma$. This would contradict the hypothesis that $\gamma_1\cup\gamma_2\cup\gamma_3$ does not contain closed geodesics.

Therefore, there exists a connected region $K_j\subset\Omega_j$ such that $\partial K_j$ is the union of $\partial\Omega_{j} = \gamma_l\cup \gamma_k$ and parts of $\sigma$ (or all of $\sigma$, if this is an embedded curve). In particular, $\partial K_j$ is the piecewise union of smooth curves, whose parts are either geodesics (parts of the immersed curve $\sigma$) or the curve $\gamma_l\cup\gamma_k$, and whose vertices have angles measuring strictly less than $\pi$. Minimize the length of the piecewise smooth closed curves contained in $K_j$ that are freely homotopic to $\gamma_{l}\cup\gamma_{k}$. Let $\beta_j$ be a minimizer. The curve $\gamma_{l}\cup\gamma_{k}$ does not intersect $\beta_j$; otherwise we would have $\beta_j=\gamma_{l}\cup\gamma_{k}$, which does not occur because $\gamma_{l}\cup\gamma_{k}$ is not a minimizer. One must analyze two cases. 

In the first, $\beta_j$ has no points in common with $\sigma$. In this situation, we can apply the second variation formula \ref{defi-quadratic-form-second-variation} and the minimization property of $\beta_j$ to conclude that it is a stable closed geodesic, contradicting the key hypothesis in the statement. It remains the case in which $\beta_j$ intersects $\sigma$. This does not happen if $\sigma$ has self-intersections, by the minimization property of $\beta_j$. Therefore, we can assume that these two curves are simple closed geodesics, and they must coincide. We conclude that $\sigma$ must minimize length among closed curves that are small perturbations of $\sigma$ contained in $K_j$. Lemma \ref{funct-exist-lem} would imply that $\sigma$ is stable, completing the proof.
\end{proof}

The conclusion of the preceding claim allows us to obtain an interesting deformation of $\gamma_l\cup \gamma_k$ by closed curves as follows.

\textbf{Claim 2}: There exists a homotopy $H^j:[0,a]\times[0,1]\to\Omega_{j}$ by closed curves, $H^j(0,s)=H^j(a,s)$ for all $s\in [0,1]$, such that $H^j(\cdot,1)=\gamma_{l}\cup\gamma_{k}$, $H^j(\cdot,0)$ is a point curve, $H^j(0,1) = H^j(a,1) = x$, $H^j(a/2,1) = y$, and $L(H^j(\cdot,s))\leq L(\gamma_{l}\cup\gamma_{k})$, for every $s \in [0,1]$. 

\begin{proof}[Proof of the Claim 2:]
    Each iteration of the Birkhoff process yields a homotopy defined on $[0,a]\times [0,1]$ between a curve in the sequence and its Birkhoff replacement. The homotopy is composed by the curves that interpolate the starting curve of each step and its geodesic approximations. Claim 1 implies that for every $r>0$, there exists a positive integer $m$ for which the closed curve obtained after $m$ iterations of the process is  a short curve contained in a geodesic ball of $(S^2,g)$ of radius $r$. For sufficiently small $r$, the radial deformation of this final curve to the center of the corresponding ball defines a homotopy of that curve that ends at a point by curves of lengths bounded away from $L(\gamma_l\cup \gamma_k)$. Concatenating the homotopies above yields $H^j$.  
\end{proof}

In the remainder of the proof, it is convenient to consider the following alternative notation. Letting $F_{j}: \mathbb{D}\to\Omega_{j}$ be the map defined on the unit disk $\mathbb{D}$ by the expression $F_j(se^{2\pi i \frac{t}{a}})=H^j(t,s)$, notice that $F_{j}(\partial \mathbb{D})=\gamma_{l}\cup\gamma_{k}$, $F_{j}\circ c_{s}=H^j(\cdot ,s)$, $F_{j}(\widetilde{x})=x$, $F_{j}(\widetilde{y})=y$, where $c_{s}(t)=se^{2\pi i \frac{t}{a}}$ is the circle of radius $s$ centered at the origin of $\mathbb{D}$, $\widetilde{x}=(1,0)$, and $\widetilde{y}=(-1, 0) \in \mathbb{D}$.

\textbf{Claim 3}:
The maps $F_j$ have degree $\pm 1$.

\begin{proof}[Proof of the Claim 3:] 
This map in only continuous, so a smooth approximation is necessary for the computation of the degree. On the other hand, $F_j$ is smooth on portions of the domain corresponding to the interpolation between a curve in $\Omega_j$ and its geodesic approximation. For instance, when performing a first geodesic approximation on the initial curve $F_j|_{c_1} = \gamma_l\cup \gamma_k$. More precisely, let $\varepsilon>0$ be such that the curves $\{F_j\circ c_s\}$, $s\in (1-\varepsilon, 1]$, correspond to the interpolation associated to the first geodesic approximation. Assume, without loss of generality, that $F_j(-1,0)=y$ is not a vertex of the partition, i.e., a neighborhood of $y$ in $\gamma_l\cup\gamma_k$ is interpolated to a minimizing geodesic joining points $\gamma_l\cup\gamma_k(t_i)\in \gamma_l$ and $\gamma_l\cup\gamma_k(t_{i+1})\in \gamma_k$, of the partition. Assume also that $y$ is not a smooth point of $\gamma_l\cup\gamma_k$. Otherwise, $x$ is not a smooth point of $\gamma_l\cup\gamma_k$, and use $x$ instead of $y$ in the argument below.

Then, $F_j$ restricted to portions of $c_s$, $s\in (1-\varepsilon, 1]$, inside a small neighborhood of $(-1,0)$ in polar coordinates is a foliation by minimizing geodesics of the open set $\tilde{\Omega}_j$ bounded by $\gamma_l\cup\gamma_k|_{[t_i, t_{i+i}]}$ and the minimizing geodesic joining its endpoints $\gamma_l\cup\gamma_k(t_i)$ and $\gamma_l\cup\gamma_k(t_{i+1})$. See Figure \ref{Fig-Birkhoff}. The local convexity of the complement of the closure of $\tilde{\Omega}_j$ in $\Omega_j$ implies that the interior points of $\tilde{\Omega}_j$ are covered exactly once by $F_j$. These points are also regular values of $F_j$. It is possible to find smooth approximations of $F_j$ for which a compact region contained in the interior of $\tilde{\Omega}_j$ is also made of regular values that are taken only once by the approximation. This completes the proof of the claim.

\begin{figure}[h]
\includegraphics[scale=.45]{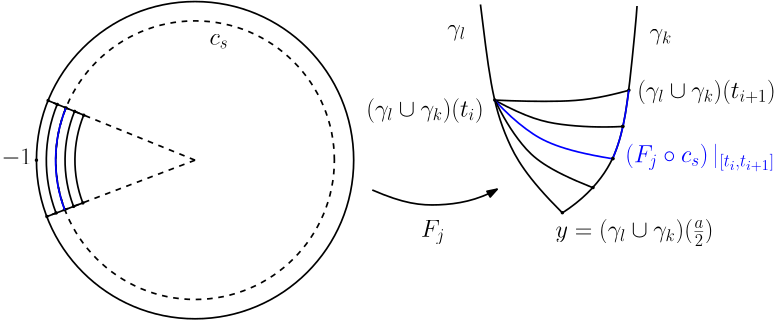}
\caption{Interpolating curves of the Birkhoff's process}
\label{Fig-Birkhoff}
\end{figure}
\end{proof}

The curves $c_s$ sweep the disk $\mathbb{D}$ out, and $F_{j}\circ c_{s}$ is a family of curves that sweep the region determined by $\gamma_{l}\cup\gamma_{k}$ out in such a way that, for $p\in c_s$ and $s(p)=dist_{\mathbb{D}}(p,0)$ the euclidean distance from $p$ to the origin,
  \begin{eqnarray}\label{eq: 3.3 }
  	dist(F_j(p),F_j(s(p),0))  \leq  \frac{L( F_{j}\circ c_{s(p)} )}{2}
   \leq  dist(x,y).
  \end{eqnarray}
  
The curves obtained by the maps $F_j$ completely sweep the regions $\Omega_{j}$. Consequently, the collection of these three families of curves cover $S^2$. Let us use $B_j$ to denote the disk domain of $F_j$, and let
$\sqcup^{3}_{j=1}B_j$ represent their disjoint union. Define the following equivalence relation: given $z\in\partial B_l$ and $w\in\partial B_k$, $k\neq l$, we identify these points, $z\sim w$, if $F_l(z)=F_k(w)$. The quotient $\sqcup^{3}_{j=1}B_j/\mathord\sim$ is homeomorphic to the two-sphere. See Figure \ref{Fig-F_j-combined}.

\begin{figure}[h]
\includegraphics[scale=.42]{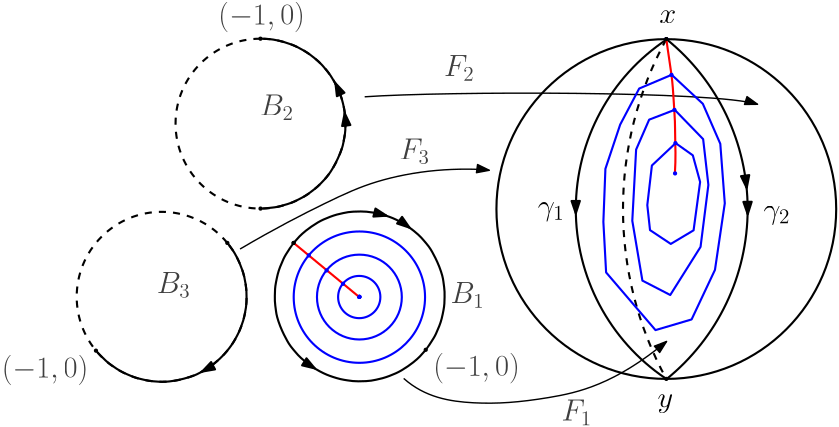}
\caption{The collected maps $F_j$. The disks $B_j$ are rotated to facilitate the visualization of the identification of their boundary portions. The blue curves represent iterations of $\gamma_1\cup\gamma_2$ under Birkhoff's process. The red curve is the evolution of the initial/final point of all those curves.}
\label{Fig-F_j-combined}
\end{figure}

We define the map $F: \sqcup^{3}_{j=1}B_j/\mathord\sim\to S^2$, given by $F(p)=F_j(p)$ for $p\in B_j$. Note that, by the equivalence relation described above, the map $F$ is well-defined. Next, we define the map $\psi: \sqcup^{3}_{j=1}B_j/\mathord\sim\to\mathcal{P}_{S^2}$ by the formula $\psi(p)= \{F(p),F(s(p),0)\}$.

\textbf{Claim 4}: $\psi\in [\Phi_{0}]$ is a sweepout with the desired properties. 

\begin{proof}[Proof of Claim 4:]
Notice that the map $p\mapsto F(s(p),0)$ has image contained in the union of the three curves $s\mapsto H^j(0,s)$, joining $H^j(0,1)=x$ and the point $H^j(0,0)$ at the end of radial deformation applied after sufficiently many iterations of Birkhoff's process in $\Omega_j$. Therefore, it is clearly possible to construct a homotopy between $p\mapsto F(s(p),0)$ and the constant map $p\mapsto x$. It follows from this property that the second entry of the map $\psi$ defined above is homotopic to a constant map. 

By Claim 3, choose an orientation on $B_j$ for which the degree of $F_j$ is one, for $j=1, 2, 3$. This choice induces orientations on the boundaries which are compatible with the gluing performed on $\sqcup_{i=1}^3 B_j$. Therefore, $F$ has degree one as a map from the two-sphere $\sqcup_{i=1}^3 B_j$ to the Riemannian sphere $(S^2, g)$. Therefore, this map is homotopic to a diffeomorphism; see Section 4.1 of \cite{Hat}. In conclusion, $\psi\in[\Phi_{0}]$, and, by expression \ref{eq: 3.3 }, $dist(\psi(p))\leq dist(x,y)$. Thus, we obtain the desired sweepout.
\end{proof}
The proof of Proposition \ref{prop-exis-swee} is complete.
\end{proof}

Proposition \ref{prop-exis-swee} provides a sweepout of $S^2$ that satisfies a property similar to the one that we need. However, this is not sufficient to conclude Theorem \ref{thm-3-simult-st}. When the additional hypothesis that $\gamma_{1}$, $\gamma_{2}$, and $\gamma_{3}$ are not simultaneously stationary is assumed, in the sense of Definition \ref{defi-simultaneously-stationary}, one must produce a sweepout whose distances are strictly below the level $dist(x,y)$, of the Riemannian distance between the endpoints of the $\gamma_i$, $i=1, 2, 3$.

\begin{prop}\label{3-simult-strict}
Let $(S^{2},g)$ be a Riemannian sphere whose closed geodesics are all unstable. Let $\{x, y\}\subset S^{2}$ be a critical point of the distance function with $x\neq y$. Suppose that $x$ and $y$ are joined by three distinct minimizing geodesics $\gamma_{i}$, $i\in\{1,2,3\}$, such that $\gamma_1\cup\gamma_2\cup\gamma_3$ does not contain closed geodesics. In addition, assume that the six angles $\theta_i$ and $\theta_i^{\ast}$ determined by those curves (see Proposition \ref{prop-exis-swee}) measure at most $\pi$, for $i\in\{1,2,3\}$. If $\{\gamma_1,\gamma_{2},\gamma_3\}$ is not simultaneously stationary, then there exists a sweepout $\psi\in[\Phi_0]$ of $S^2$ such that $dist(\psi(p))<dist(x,y)$ for every $p\in S^2$.
\end{prop}

\begin{proof}
The hypotheses in the statement imply that Proposition \ref{prop-exis-swee} can be applied. Then, we obtain a sweepout $\psi = (\psi_1, \psi_2) : S^{2}\to\mathcal{P}_{S^{2}}$ such that $dist(\psi(p))\leq dist(x,y)$, for all $p\in S^{2}$. Recall that, along the construction of $\psi$, the points $\psi_1(p)$ and $\psi_2(p)$ belong to a piecewise smooth closed curve $\Gamma_p: t\in [0, a]\mapsto H^j(t, s(p))$ obtained from an iteration of Birkhoff's process starting with $H^j(\cdot, 1)$, a parametrization of $\gamma_l\cup\gamma_k=\partial\Omega_j$, for some $j=1, 2, 3$, for $s(p)$ outside of a small neighborhood of $s=0$. Recall that $\Omega_j$ are the connected components of $S^2\setminus\{\gamma_1, \gamma_2, \gamma_3\}$. For $s(p)$ close to zero, the curves are all very short, and are obtained by a radial deformation. 

Proposition \ref{prop-exis-swee} also implies that $dist(\psi(p)) = dist(x,y)$ if, and only if, $\psi(p)=\{x,y\}$. In this case, one must have $\Gamma_p = \gamma_l\cup \gamma_k$, for some $k$ and $l$. Indeed, the curves $H^j(\cdot, s(p))$ different from $\gamma_l\cup \gamma_l$ all have length strictly lower than $2\cdot dist(x,y)$. Let us use $K$ to denote the set of points $p\in S^2$ at which $\psi(p)=\{x,y\}$. This set is clearly compact. Moreover, the Birkhoff's process gives us that the first variation of the length of $\Gamma_p$ in the direction of a fixed vector field is continuous on $p \in S^2$.

Since $\{\gamma_{i}\}$, $i\in\{1, 2, 3\}$, is not simultaneously stationary, there exist vectors $v_1\in T_x S^2$ and $v_2\in T_y S^2$ such that $\left\langle (v_1,v_2),\nu_{\gamma_i}\right\rangle <0$, for $i=1, 2, 3$, where $\nu_{\gamma_i}=(-\gamma_{i}^{\prime}(0),\gamma_{i}^{\prime}(a))$. Extend these two vectors to a vector field $X$ on $S^2$ whose support is contained in a small neighborhood of the set containing $x$ and $y$ only. Let $\phi: [0,\infty)\times S^2\to S^2$ be the unique maximal flow generated by $X$.
	
We claim that there exists an open set $U\subset S^2$, $\theta>0$, and $\overline{s}_0>0$ such that $K\subset U$, and the first variation of length in the direction of $X$ satisfies $\delta \phi_{\overline{s}}(\Gamma_p) (X) < -\theta <0$, for every $0\leq \overline{s} < \overline{s}_0$ and $p$ in the closure of $U$, where $\delta \phi_{\overline{s}}(\Gamma_p) (X)$ represents the first derivative of the length functional at the curve $\phi_{\overline{s}}(\Gamma_p)$, in the direction of $X$. In order to prove the claim, notice that, in a neighborhood of $K$, the curves $\Gamma_p$ have the form $\Gamma_p = \cup_{i=0}^{N-1} [q_i(p), q_{i+1}(p)]$ of the union of $N$ geodesic segments with endpoints $q_i(p)$ and $q_{i+1}(p)$, which vary continuously with respect to $p$. Then, $\phi_{\overline{s}}(\Gamma_p)$ is the union of the smooth curves $\phi_{\overline{s}} ([q_i(p), q_{i+1}(p)])$. The derivatives of $t \mapsto\phi_{\overline{s}} \circ [q_i(p), q_{i+1}(p)](t)$ with respect to the variable $t$ depend continuously on $p$. In particular, $\delta \phi_{\overline{s}}(\Gamma_p) (X)$ is continuous on $\overline{s}$ and $p$ simultaneously. Finally, since, for every $p\in K$, we have $\Gamma_p = \gamma_l\cup \gamma_k$, it follows that $\delta \phi_{\overline{s}}(\Gamma_p) (X) < 0$, for $\overline{s}=0$ and $p\in K$. The claim follows by continuity.

Let $\eta : S^2 \rightarrow [0,1]$ be a smooth function supported in $U$ and such that $\eta$ is constant equal to one in a compact set $K^{\prime}$ for which $K\subset \text{int}(K^{\prime})$. The map that associates to each $(\overline{s}, p) \in [0, \overline{s}_0/2]\times S^2$ the pair of points $\{\phi_{\eta(p)\overline{s}} (\psi_1(p)), \phi_{\eta(p)\overline{s}} (\psi_2(p))\}$ is a homotopy between sweepouts $\psi$ and $\tilde{\psi}$, corresponding to $\overline{s}=0$ and $\overline{s} = \overline{s}_0/2$, respectively.

Notice that if $p \notin U$, then $\eta(p)=0$, $\tilde{\psi}(p) = \psi(p)$, and $\max_{p\notin U} dist(\psi(p)) < dist(x,y)$. Suppose now that $p\in U$. In this case, the points of $\tilde{\psi}(p)$ belong to the curve $\phi_{\eta(p)\overline{s}_0/2}(\Gamma_p)$, which has length at most $L(\Gamma_p)$. When $p\in K^{\prime}$, $\eta(p)=1$, and the length of $\phi_{\eta(p)\overline{s}_0/2}(\Gamma_p)$ is uniformly lower than $L(\Gamma_p)$. In particular, it is uniformly lower than  $2\cdot dist(x,y)$. If $p\in U\setminus K^{\prime}$, then $L(\Gamma_p)$ is already uniformly lower than $2 \cdot dist(x,y)$. In any case, $dist(\tilde{\psi}(p)) \leq L(\phi_{\eta(p)\overline{s}_0/2}(\Gamma_p))/2$, is always uniformly lower than $dist(x,y)$.
\end{proof}

\begin{rmk}\label{rmk-more-minimizing-geodesics} 
The methods and results of Propositions \ref{prop-exis-swee} and \ref{3-simult-strict} remain valid for sets of $k>2$ minimizing geodesics with endpoints $x$ and $y$.
\end{rmk}

The next proposition has a proof similar to that of Proposition \ref{prop-exis-swee}. It applies, in particular, to a pair of minimizing geodesics whose endpoints are a nontrivial critical point of the distance, forming a closed geodesic.

\begin{prop}\label{exis-sweep-closed-geod}
	
	Let $(S^{2},g)$ be a Riemannian sphere which does not contain stable closed geodesics. Let $\{x,y\}\subset S^{2}$ be a critical point of the distance function with $x\neq y$. Suppose $x$ and $y$ are joined by two minimizing geodesics $\gamma_1$ and $\gamma_2$ such that $\gamma_1\cup\gamma_2=\gamma$ is a closed geodesic. Then, there exists a sweepout $\psi\in[\Phi_0]$ of $S^2$ such that $dist(\psi(p))\leq dist(x,y)$ for every $p\in S^2$.

    Moreover, the sweepout $\psi$ is of the form $\psi(p) = \{\psi_1(p), \psi_2(p)\}$, where $\psi_1(p)$ and $\psi_2(p)$ belong an element of a one parameter family of closed curves of length at most $L(\gamma)$, whose $\gamma$ is the only element of length $L(\gamma)$.
\end{prop}

\begin{proof}
By hypothesis, $\gamma$ is not stable. Lemma \ref{funct-exist-lem} gives us a vector field $V$ on $S^2$, whose flow $\rho$ satisfies: $\{\rho_s(\gamma)\}$, $s\in [-\varepsilon, \varepsilon]$, is a smooth family of smooth closed curves such that $L(\rho_s(\gamma)) \leq 2\cdot dist(x,y)-\theta\cdot s^2$, for some $\theta>0$. Recall that all $\rho_s(\gamma)$ belong to the same component of $S^2\setminus \{\gamma\}$, for $s \in (0, \varepsilon]$. The same holds for $s\in [-\varepsilon, 0)$. It follows that $dist(p_s, q_s)\leq L(\rho_s(\gamma))/2 \leq dist(x,y)$, for every $p_s,q_s \in \rho_s(\gamma)$, and the strict inequality holds whenever $s\neq 0$. For $p\in \gamma$, define $\psi(p)=\{p,x\}$. On a small neighborhood of $\gamma$ define $\psi(p)=\{p,\rho_{s}(x)\}$, where $s\in [-\epsilon,\epsilon]$ is the only number in this interval such that $p\in\rho_{s}(\gamma)$. Hence, $dist(\psi(p))  \leq  dist(x,y)$, for every $p\in \rho_s(\gamma), s \in [-\epsilon,\epsilon]$. The strict inequality occurring whenever $s\neq 0$. 

Label the two closed regions of the sphere $S^2$ determined by $\gamma$ by $U_{1}$ and $U_{2}$, in such a way that $\rho_s(\gamma)\subset U_{1}$, for $s<0$, and $\rho_s(\gamma)\subset U_{2}$, for $s>0$. Since $\gamma$ is a closed geodesic, both regions are locally convex. 

Furthermore, we can also assume that the curves $\rho_s(\gamma)$, $s\in [-\varepsilon, \varepsilon]$ and $s\neq 0$, are also locally convex with respect to the regions determined by them that do not contain $\gamma$. This is a consequence of a particular choice of $u$ on Lemma \ref{funct-exist-lem} before extending the vector field along $\gamma$ to $S^2$. Assume that the restriction of $V$ to $\gamma$ is an eigenvector of the Jacobi operator of $\gamma$ associated to a negative eigenvalue. The claim follows from the formula for the derivative of the geodesic curvature on a smooth deformation of a curve, decreasing $\varepsilon>0$, if needed. See the appendix of \cite{Amb2013}, for instance. 

Thus, Birkhoff's process, described in the proof of Proposition \ref{prop-exis-swee}, starting at either $\rho_{\pm\varepsilon}(\gamma)$ remains in the respective region $U_{\pm\varepsilon}$ determined by the initial curve and that does not contain $\gamma$. The processes starting at $\rho_{\pm\varepsilon}(\gamma)$, similarly to the arguments in Proposition \ref{prop-exis-swee}, end at points (using radial deformations of a sufficiently short iterate of the initial curve). 

Then, there exist degree one continuous maps $F_{-} : \mathbb{D} \rightarrow U_{-\varepsilon}$ and $F_{+} : \mathbb{D} \rightarrow U_{\varepsilon}$ such that $F_{\pm}(\partial \mathbb{D}) = \rho_{\pm\varepsilon}(\gamma)$, $L(F_{\pm} \circ c_r) \leq L(\rho_{\pm\epsilon}(\gamma)) < L(\gamma)$, for all $r \in [0,1]$, where $c_r$ is the circle of radius $r$ centered at the origin of $\mathbb{D}$.

Recall that the region bounded by the curves $\rho_{-\varepsilon}(\gamma)$ and $\rho_{\varepsilon}(\gamma)$ was swept out by the image of $\gamma$ under $\rho$. Denote this region by $T$. In order to describe the desired sweepout by pairs of points, let us start by fixing a parametrization of the closed geodesic $\gamma$ on the interval $[0,a]$, with $\gamma(0)=\gamma(a)=x$, and letting $f :[0,a]\times [-\varepsilon, \varepsilon]\rightarrow S^2$ be given by $f(t, s) = \rho_s(\gamma(t))$. Since $f(0, s) = f(a, s)$, for every $s$, the map $f$ descends to the quotient $S^1\times [-\varepsilon, \varepsilon]$ of the domain of $f$ by the equivalence relation $(0,s)\sim (a,s)$. Suppose, without loss of generality, that $F_{\pm}(1,0)=f(0,\pm \varepsilon)=\rho_{\pm\varepsilon}(x)$.

The domains of the maps $F_{\pm}$ are disks and the domain of $f$ is an annulus. The disjoint union of the different domains is considered with the equivalence relation between boundary points: $p\sim [(t,-\varepsilon)]$ if $F_{-}(p) = f(t, -\varepsilon)$, and $p\sim [(t,\varepsilon)]$ if $F_{+}(p) = f(t, \varepsilon)$. The quotient is a topological sphere $S^2$, on which we can define a map $\psi$ as in the first paragraph of this proof for equivalence classes of points in $T$, and by $\psi([p]) := \{F_{\pm}(p), F_{\pm}(r(p),0)\}$, whenever $p$ belongs to the domain of $F_{\pm}$. Here $r(p) = |p|$ is the distance from $p$ to the center of the respective disk. Notice that $\psi([(t,s)]) = \{f(t,s), f(0,s)\} = \{\rho_s(\gamma(t)), \rho_s(x)\}$. Arguing as in the proof of Proposition \ref{prop-exis-swee}, we verify that $\psi$ is a sweepout satisfying all the desired properties.
\end{proof}

\begin{rmk}\label{rmk-prop-generaliz-closed-geodesics}
The argument in the proof of Proposition \ref{exis-sweep-closed-geod} works even if the $\gamma_i$ whose union is a closed geodesic are not minimizing. The obtained sweepout satisfies $dist(\psi(p))\leq \frac{L(\gamma)}{2} = L(\gamma_1) = L(\gamma_2)$, for every $p\in S^2$.
\end{rmk}

The last ingredient needed in the proof of the main theorem is the convex geometry result, Lemma \ref{lemm-convex-geometry}, which is proved in Section \ref{sect.convex.geom}. We are now ready to present the proof of Theorem \ref{thm-3-simult-st}.

\begin{proof}[Proof of Theorem \ref{thm-3-simult-st}]
By Propositions \ref{prop-characterization-critical-simult-st} and \ref{prop-exis-swee}, and Remark \ref{rmk-more-minimizing-geodesics}, every critical point $\{x,y\}\subset S^2$ of $D$ with $x \neq y$ is such that there exists a sweepout $\psi\in[\Phi_0]$ satisfying $dist(x,y)=\max_{p \in S^2} dist(\psi(p))$. By the definition of $W_d(S^2,g)$, the inequality $W_d(S^2,g)\leq dist(x,y)$ follows immediately. 

Let $\{x,y\}$ be a critical point with $dist(x,y)=W_d(S^2,g)>0$. Proposition \ref{prop-characterization-critical-simult-st} implies that there exists an integer $k\leq 5$ and a set of $k$ simultaneously stationary minimizing geodesics with endpoints $x$ and $y$. If $k\in \{2, 3\}$, then the proof is complete. Let us analyze the cases $k=4$ and $k=5$ separately. 

Suppose, by contradiction, that there exists a set $\{\gamma_1, \ldots, \gamma_k\}$ of simultaneously stationary minimizing geodesics with endpoints $x=\gamma_i(0)$ and $y=\gamma_i(a)$, $k\in \{4,5\}$, but no proper subset is simultaneously stationary. Let $v_i = \gamma_i^{\prime}(0) \in T_{x}S^2$ and $w_i = -\gamma_i^{\prime}(a)$, $1\leq i\leq k$. Fix an orientation on $S^2$. Up to relabeling, assume that $v_1, \ldots, v_k$ are disposed in this order as one moves around the unit circle of $T_x S^2$ in the clockwise (negative) direction. Since the minimizing geodesics do not intersect, it follows that the vectors $w_1, \ldots, w_k$ appear in this order in the counterclockwise (positive) direction of $T_y S^2$. By Definition \ref{defi-simultaneously-stationary}, there exist $\{ \lambda_i\}$, $1\leq i \leq k$, $\lambda_i \in (0,1)$ and $\sum_{i=1}^k \lambda_i =1$, and such that $\sum_{i=1}^k \lambda_i (v_i, w_i) = (0,0)$.

Let us begin with the case $k=4$. Notice that Lemma \ref{lemm-convex-geometry} ensures that there exists a choice of three indices $A\subset \{1, 2, 3, 4\}$, $\#A =3$, such that the origin belongs to the two triangles with vertices at the points $\{v_i : i \in A\}$ or $\{w_i : i \in A\}$. Hence, we obtain three distinct minimizing geodesics that satisfy the angle conditions of Proposition \ref{3-simult-strict}. Then, this result implies that there exists a sweepout $\psi$ such that $dist(\psi(p))<dist(x,y)$, for all $p\in S^2$. However, this is an absurd, because it contradicts the definition of $W_d(S^2,g)$. This concludes the proof in the case $k=4$.

Assume now that $k=5$. We claim that there exists a choice of four indices $A\subset \{1, 2, 3, 4, 5\}$, $\#A =4$, such that the origin belongs to the two quadrilaterals with vertices at the points $\{v_i : i \in A\}$ or $\{w_i : i \in A\}$. Label $\theta_i$ the angles between $v_i$ and $v_{i+1}$ that do not contain the other vectors $v_j$, and $\theta_5$ the angle between $v_5$ and $v_1$. The origin belongs to the quadrilateral with vertices at $v_1$, $v_2$, $v_3$, and $v_4$ if, and only if, $\theta_4+\theta_5 \leq \pi$. In case $\theta_4+\theta_5>\pi$, it follows that $\theta_1+\theta_2<\pi$ and $\theta_2+\theta_3 <\pi$. In addition, either $\theta_5+\theta_1$ or $\theta_3+\theta_4$ is strictly lower than $\pi$. In conclusion, the quadrilaterals with vertices at $\{v_1, v_3, v_4, v_5\}$ or at $\{v_1, v_2, v_4, v_5\}$ contain the origin, and at least one of the quadrilaterals with vertices at $\{v_2, v_3, v_4, v_5\}$ or at $\{v_1, v_2, v_3, v_5\}$ also contains the origin. 

Hence, we verified that at least three of the quadrilaterals with consecutive vertices in one of the lists $\{v_2,v_3,v_4,v_5\}$, $\{v_1,v_3,v_4,v_5\}$, $\{v_1,v_2,v_4,v_5\}$, $\{v_1,v_2,v_3,v_5\}$, or $\{v_1,v_2,v_3,v_4\}$ contain the origin. The same holds for the vectors $w_i$ at $T_y S^2$. In particular, up to relabeling, assume that the origin belongs to the quadrilaterals $\{v_1, v_2, v_3, v_4\}$ and $\{w_1, w_2, w_3, w_4\}$. This implies that the angles determined by the velocity vectors of $\gamma_1$, $\gamma_2$, $\gamma_3$, and $\gamma_4$ at $x$ and $y$ all measure at most $\pi$.  Since $\{\gamma_1, \gamma_2, \gamma_3, \gamma_4\}$ is assumed not to be simultaneously stationary, the methods of Proposition \ref{3-simult-strict}, see Remark \ref{rmk-more-minimizing-geodesics}, give us a sweepout $\psi$ such that $\sup_{p\in S^2} dist(\psi(p))<dist(x,y)$.
\end{proof}

\subsection{Example: The min-max width associated to the distance function of the Calabi-Croke sphere}\label{subsect-gluing-2-triangles}

In this section, we present examples of positively curved Riemannian spheres such that a critical pair $\{x,y\}$ of the distance function that realizes the intrinsic min-max width, $dist(x,y)=W_d(S^2, g)$, is joined by three distinct simultaneously stationary minimizing geodesics. We begin with the following associated example. The Calabi-Croke sphere is obtained by gluing two copies of a flat equilateral triangle with sides of unit length along their boundaries, $Cf$. \cite{CC}. Notice that the gluing is smooth away from the vertices of the triangles, but it creates three conical singularities $x_1$, $x_2$, and $x_3$ corresponding to cone angles measuring $2\pi/3$. Conjecturally, the Calabi-Croke metric attains the supremum of the length of a shortest closed geodesic among all unit area Riemannian metrics on $S^2$. We refer the reader to \cite{Bal}, \cite{CC}, and \cite{SS}.

\begin{figure}[h]
    \centering
    \includegraphics[width=0.6\linewidth]{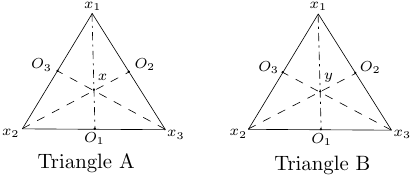}
    \caption{Triangular pieces of the Calabi-Croke sphere.}
    \label{fig:exam-three-simul-sta}
\end{figure}

Let $A$ and $B$ be the two copies of the side one equilateral triangles used to construct the Calabi-Croke sphere. By drawing the bisector of each angle formed at the vertices $x_1$, $x_2$, and $x_3$, we obtain the barycenters $x$ of triangle $A$, and $y$ of triangle $B$. See Figure \ref{fig:exam-three-simul-sta}. Note that the segments connecting the barycenters $x$ and $y$ of each triangle to the midpoints $O_i$ of each side give us three minimizing geodesics $\gamma_i = xO_i \cup O_i y$, $i=1, 2, 3$. Since these geodesics meet at $x$ and $y$ at $2\pi/3$ angles, it follows that $\{\gamma_1, \gamma_2, \gamma_3\}$ is simultaneously stationary, according to Definition \ref{defi-simultaneously-stationary}. In particular, Proposition \ref{prop-characterization-critical-simult-st} implies that $\{x,y\}$ is a critical point of the distance function. 

There are smooth metrics with non-negative curvature on the two-sphere that approximate the Calabi-Croke metric in such a way that they agree with the original metric outside an arbitrarily small neighborhood of the vertices. This can be achieved by isometrically embedding a portion $B(r)$ of the model cone into Euclidean 3-space as the graph of $z=f(x,y) = 2\sqrt{2}\sqrt{x^2+y^2}$ over the ball of radius $r < 1/3$ in the $xy$-plane. The points in $B(r)$ are those at distance at most $3r$ from one of the vertices $x_i$.  

Let $\{z = f_m(x,y)\}$, $m\in \mathbb{N}$, be a sequence of convex smooth radial functions defined on the ball of radius $r$ centered at the origin of the $xy$-plane, whose graphs agree with the graph of $f$ in a neighborhood of $\partial B(r)$, and such that $f_n$ uniformly converges to $f$ everywhere. Assume also that $f_n$ smoothly converges to $f$ on compact sets not containing the origin. The Gauss-Bonnet Theorem implies that the geodesic curvature of the boundary of rotationally symmetric balls of the approximating graphs are positive constants. We can also assume that the diameter of the replacement graphs of the $f_m$ tend to zero uniformly as $r$ decreases to zero. See more details about this construction in Lemma 4.1 of \cite{ManSch}.

Replace the three copies of $B(r)$ on the Calabi-Croke sphere by the graph of $f_m$ to obtain smooth metrics $h_m$ on $S^2$ of non-negative curvature which approximate the Calabi-Croke metric. Smoothly perturb these metrics to positively curved 2-spheres. More precisely, let us use $(S^2, g_m)$, $m\in \mathbb{N}$, to denote the metric corresponding to the time $t=\delta_m>0$ of the Ricci flow starting on $h_m$, for large $m$, where $\delta_m$ are sufficiently small to ensure that the corresponding flows exist beyond that time and $\delta_m \rightarrow 0$. The metrics $g_m$ are positively curved by the maximum principle. Moreover, they smoothly converge to the Calabi-Croke sphere in compact sets not containing the vertices. Making $\delta_m$ smaller, if necessary, one can assume that closed geodesics of $g_m$ can not be contained in the regions that replaced $B(r)$.

By Theorem \ref{thm-3-simult-st}, the min-max widths $W_d(S^2, g_m)$, $m\in \mathbb{N}$, are realized by a critical point $\{x_m, y_m\}$ of the distance function that is joined either by three simultaneously stationary minimizing geodesics, or by a pair of minimizing geodesics $\{\sigma_i^m\}$, $i=1, 2$, such that $\sigma_1^m\cup \sigma_2^m$ is a closed geodesic. 

We claim that the case of two minimizing geodesics does not occur for large $m$. Suppose, by contradiction, that $W_d(S^2, g_m) = dist_{g_m}(x_m, y_m)$ and that there exists a pair of minimizing geodesics $\{\sigma_i^m \}$, $i=1, 2$, with endpoints $x_m$ and $y_m$ such that $\sigma^m = \sigma_1^m\cup \sigma_2^m$ is smooth, for infinitely many values of $m$. It is known that $\sigma^m$ is not contained in the neighborhoods in $S^2$ of the points associated to vertices of the Calabi-Croke metric that correspond to $B(r)$-type regions. Let $K(r)$ denote the compact region that is obtained as the complement of the interiors of the three $B(r)$-type regions. Therefore, a sequence of points of $\sigma^m$ converges to a point $p\in K(r)$.

Up to subsequence, $\sigma^m\cap K(r)$ smoothly converges to a curve $\overline{\sigma} \subset S^2$ that is made of geodesic line segments of the Calabi-Croke sphere. This is a consequence of the smooth dependence of solutions of ordinary differential equations on initial conditions and parameters. Moreover, the fact that the $\sigma^m$ are closed curves yields that one the following holds: either (i) $\overline{\sigma}$ is closed; or (ii) $\overline\sigma$ has a line segment with endpoints at the boundary of the same $B(r)$-type region; or (iii) $\overline{\sigma}$ has at least two line segments with endpoints at boundaries of $B(r)$-type regions. 

Given $\varepsilon >0$, we can choose $r$ sufficiently small such that each geodesic line segment of the Calabi-Croke sphere contained in $K(r)$ and with endpoints in different components of the boundary of this region has length at least $1-\varepsilon$. This is a consequence of the fact that $dist(x_i, x_j) = 1$, for $i\neq j$. Similarly, we can suppose that a geodesic line segment contained in $K(r)$ and with endpoints in a single boundary component $\partial B(r)$ of this region has length at least $\sqrt 3 -\varepsilon$. Indeed, a curve with these properties must necessarily leave the open ball centered at the same vertex of $B(r)$ and of radius $\sqrt{3}/2$, the height of the equilateral triangle. This is a consequence of the maximum principle, since the set defined as the difference of these concentric balls is foliated by curves of positive geodesic curvature starting at $\partial B(r)$. Hence, for small $r$, $\overline{\sigma}$ has at least two segments of length at least $(\sqrt{3}-\varepsilon)/2$. Therefore, in cases (ii) or (iii), the length of $\sigma$ is at least $\sqrt{3}-\varepsilon$. For case (iii), we need $\sqrt 3 - \varepsilon < 2(1- \varepsilon)$, or simply $\varepsilon < 2- \sqrt 3$. 

Let us consider case (i). Through any point $p$ in the Calabi-Croke sphere that is not a conical singularity, there are exactly three closed geodesics of length $\sqrt{3}$, with one point of self-intersection (an eight-figure). These are the curves that emanate from $p$ in the directions of the vectors that are orthogonal to the sides of the equilateral triangle that contains $p$. Moreover, these are shortest closed geodesics. See \cite{CalabiCao} and \cite{Beach}, for instance. This fact yields $L(\overline\sigma) \geq \sqrt{3}$.

In all cases, the contradiction hypothesis implies that $L(\sigma^m) \geq \sqrt{3}-\varepsilon$, for fixed small $r$ and large $m$ for which the critical point $\{x_m, y_m\}$ realizing the width bounds a pair of minimizing geodesics as above. Notice that we use the estimate on the length of $\overline\sigma$ as an asymptotic lower bound for the lengths of the closed curves $\sigma^m$ only. In particular, $W_d(S^2, g_m) \geq (\sqrt{3}-\varepsilon)/2$  

The map $\psi: p\in S^2 \mapsto \{p,x\} \in \mathcal{P}_{S^2}$ is clearly a sweepout, where $x$ denotes the barycenter of triangle $A$. For the Calabi-Croke metric $g^{\ast}$, we have
$\sup_{p\in S^2} dist_{g^{\ast}}(\psi(p)) = dist_{g^{\ast}}(x,y) = \sqrt{3}/3$. The same sweepout can used to estimate the min-max width associated to the distance function of the approximating metrics as $W_d(S^2, g_m) \leq \sqrt{3}/3+\varepsilon$, for large $m \in \mathbb{N}$. It suffices to show that $\sqrt{3}/3+\varepsilon$ is an upper bound for the maximal distance to $x$ with respect to $g_m$. This can be achieved by making $r$ and the $\delta_m$ smaller, if necessary. Indeed, the diameters of the replacements of the $B(r)$-type regions tend to zero as $r$ decreases to zero. In particular, if $\varepsilon< \sqrt{3}/9$, it follows that 
$W_d(S^2, g_m) < (\sqrt{3}-\varepsilon)/2$, which contradicts the fact that $W_d(S^2, g_m) \geq (\sqrt{3}-\varepsilon)/2$ proved above. We conclude that, for large $m$, the widths of $(S^2, g_m)$ are realized by critical pairs $\{x_m, y_m\}$ joined by three simultaneously stationary minimizing geodesics.

\subsection{Example: removing the hypothesis on the non-existence of stable closed geodesics}\label{exam-critical-value-below-width}

In this section we give an example that shows that the hypothesis on the non-existence of stable closed geodesics can not be removed from Theorem \ref{thm-3-simult-st}. More precisely, we observe that the (intrinsic) width $W_d(\Sigma)$ of a certain rotationally symmetric metric on $\Sigma = M^2 = S^2$ with a thin neck is strictly bigger than $dist(x,y)$, where $\{x,y\}$ is a critical point of the distance function which bounds two minimizing geodesics $\gamma_1$ and $\gamma_2$ whose union $\gamma_1\cup \gamma_2$ is a closed locally minimizing geodesic representing the shortest curve among those contained in the neck which are homotopically non-trivial on that region. The construction is as follows.

Consider two copies of the unit two-sphere in $\mathbb{R}^3$ with a small geodesic ball removed. Glue these regions using a rotationally symmetric neck. In order to formalize the construction, we use positive parameters $\varepsilon$ and $a>0$, to be chosen. Consider the surface of revolution $\Sigma$ parametrized by $\textbf{x}(u,v) = (u, f(u) \sin v, f(u)\cos v)$, where $f=f(u)$ is a non-negative function defined on the interval $[-a-2+\varepsilon, a+2-\varepsilon]$, satisfying the following properties:
\begin{itemize}
    \item $f(u) = \sqrt{1-(u-(-a-1+\varepsilon))^2}$, for $-a-2+\varepsilon \leq u \leq -a$;
    \item $f$ is decreasing on $[-a,0]$ and increasing on $[0, a]$;
    \item $f(u) = \sqrt{1-(u-(a+1-\varepsilon))^2}$, for $a \leq u \leq a+2-\varepsilon$.
\end{itemize}
The rotation about the y-axis of the portions of the graph of $f$ over the intervals $[-a-2+\varepsilon, -a]$ and $[a, a+2-\varepsilon]$ generate the spherical regions with a ball removed. And the rotation of the portion over $[-a,a]$ generates the region that is referred to as the neck. Since $f$ has a local minimum at $u=0$, the circle $\gamma(v) = \textbf{x}(0,v)$ is a closed geodesic. Consider $x = \gamma(0)$ and $y=\gamma(\pi)$, and let $\gamma_1$ and $\gamma_2$ denote the two portions of $\gamma$ with endpoints $x$ and $y$. Let us point out some basic geometric properties of these objects.

\textbf{Claim:} The length of $\gamma_i$ is bounded by $\pi\sqrt{1-(1-\varepsilon)^2}$, $i=1,2$. The length of any curve starting at a point of $\gamma$ that leaves the neck region is bigger than $a$. If, in addition, $a> \frac{\pi}{2}\sqrt{1-(1-\varepsilon)^2}$, then $\gamma_1$ and $\gamma_2$ are minimizing geodesics of $\Sigma$ joining $x$ and $y$. In particular, $\{x,y\}$ is a critical point of the distance function.

\begin{proof}[Proof of the Claim:]
We clearly have $f(0) < f(-a) = f(a) = \sqrt{1-(1-\varepsilon)^2}$. Hence, $L(\gamma) = 2\pi f(0)< 2\pi \sqrt{1-(1-\varepsilon)^2}$, and $L(\gamma_i) < \pi \sqrt{1-(1-\varepsilon)^2}$.

Let $\alpha(t) = \textbf{x}(u(t), v(t))$, $t\in [0,1]$, be a curve such that $\alpha(0)$ belongs to the boundary of the neck region, and $\alpha(1) \in \gamma$. Then, $\beta(t) = \textbf{x}(u(t), v(0))$, $t\in [0,1]$, also connects the boundary of the neck to $\gamma$ and it is not hard to verify that
$$
L(\alpha)\geq L(\beta) = \int_0^1 |u^{\prime}(t)|\sqrt{1+(f^{\prime}(u(t))^2}\, dt > |u(1)-u(0)| = a.
$$
The strict inequality above occurs because $f^{\prime}$ does not vanish identically on $[-a,0]$ nor $[0,a]$.

Let us prove that $\gamma_1$ and $\gamma_2$ are minimizing geodesics of $\Sigma$ joining $x$ and $y$. Assume that $a> \frac{\pi}{2}\sqrt{1-(1-\varepsilon)^2}$. By the fact proved in the preceding paragraph, any curve joining $x$ and $y$ that leaves the neck region has length at least $2a$. Hence, its length is strictly bigger than $\pi\sqrt{1-(1-\varepsilon)^2}$. In particular, it is longer than $\gamma_i$. Therefore, any minimizing geodesic of $\Sigma$ joining $x$ and $y$ is contained in the neck region. The fact that $f$ is decreasing on $[-a,0]$ and increasing on $[0, a]$ implies, by Clairaut's relation, that $\gamma_1$ and $\gamma_2$ are the only geodesics joining $x$ and $y$ and contained in the neck region. In particular, since $\gamma_1\cup \gamma_2$ is a closed geodesic, the set $\{\gamma_1, \gamma_2\}$ is simultaneously stationary with endpoints $x$ and $y$, and, by Proposition \ref{prop-characterization-critical-simult-st}, $\{x,y\}$ is a critical point of the distance function. This completes the proof.
\end{proof}

Let $\Omega_1$ denote the closed region of $\Sigma$ which corresponds to one of the spherical regions with a small ball removed. We use $B$ to denote the open ball that was removed from the unit sphere to form $\Omega_1$. And let $\Omega_2$ be the open region defined by $\Omega_2 = \Sigma\setminus \Omega_1$. Let $F: S^2 \rightarrow \Sigma$ be a homeomorphis from the unit sphere that is the identity on $\Omega_1$ (or, more precisely, a translation). Notice that $F$ maps $B$ homeomorphically onto $\Omega_2$.

The argument in the proof of the positivity of the min-max width associated to the distance function, included in the proof of Theorem \ref{thm-W_d-critical}, only requires that the map $i$ in that proof is a homeomorphism and can be applied to the mapping $F$ above to yield
$$
W_d(\Sigma) \geq \inf \{dist(F(x), F(-x)) : x\in S^2\}.
$$

We claim that $dist(F(x), F(-x)) \geq \pi-2\arccos(1-\varepsilon)$, for all $x\in S^2$. Let us consider two cases. In the first case, $F(x), F(-x) \in \Omega_1$. If a minimizing geodesic of $\Sigma$ joining $F(x)$ and $F(-x)$ is contained in $\Omega_1$, then it is half of a great circle and $dist(F(x), F(-x))=\pi$. Otherwise, the minimizing geodesic joining those points intersects $\Omega_2$, and, as a consequence, its length is at least the sum of the spherical distance between $F(x)$ and $B$, and the distance between $F(-x)$ and $B$. One easily sees that this sum of distances equals $\pi-2r$, where $r=\arccos(1-\varepsilon)$ is the spherical radius of $B$. In the second case, either $F(x)$ or $F(-x)$ belong to $\Omega_2$. Hence, the other point belongs to $-B$, and the distance between those points on $\Sigma$ is bounded from below by the spherical distance between $B$ and $-B$, which equals $\pi-2r$. The claimed fact is proven.   

Putting everything together, it follows that 
$$
dist(x,y)=L(\gamma_i) < \pi\sqrt{1-(1-\varepsilon)^2}\leq \pi-2\arccos(1-\varepsilon) \leq W_d(\Sigma),
$$
where the second inequality holds whenever $0< \varepsilon\leq \varepsilon_0$, where $\varepsilon_0\approx 0.195987$ is the root of the equation $\pi\sqrt{1-(1-\varepsilon_0)^2} = \pi-2\arccos(1-\varepsilon_0)$.


\subsection{Proof of Theorem \ref{thm-index-bound}}

In this section we prove the claim about the index of the closed geodesic made of two simultaneously stationary minimizing geodesics. We begin with the following key proposition.

\begin{prop}\label{2-clo-stri}
Let $(S^{2},g)$ be a Riemannian sphere which does not contain stable closed geodesics. Let $\{x,y\}$ be a critical point of the distance with $x\neq y$. Suppose that $x$ and $y$ are joined by two distinct minimizing geodesics $\gamma_1$ and $\gamma_2$ such that $\gamma_1\cup\gamma_2=\gamma$ is a closed geodesic of index $\geq 2$. Then, there exists sweepout $\psi\in[\Phi_0]$ such that $dist(\psi(p))< dist(x,y)$, for all $p\in S^2$.
\end{prop}

\begin{proof}
Proposition \ref{exis-sweep-closed-geod} gives us a sweepout $\psi$ with $dist(\psi(p))\leq dist(x,y)$, for all $p\in S^{2}$. In the construction, a family of curves $\rho_s(\gamma)$, $s\in [-\varepsilon,\varepsilon]$, for which the smooth function $L(\rho_s(\gamma))$ has a critical point at $s=0$, and a strictly negative second derivative at that point. Recall that $\rho_s$ was obtained using the flow of a vector field $V$ on $S^2$ whose restriction to $\gamma$ is a nowhere-vanishing eigenvector field of the Jacobi operator with negative eigenvalue.

Since the Morse index of $\gamma$ is at least two, the maximum dimension of a vector space of vector fields normal to $\gamma$ where $\mathcal{Q}$ is negative definite is at least two. Choose a normal vector field $Z$ in this two-dimensional subspace such that $Z$ and $V$ are linearly independent, and $\mathcal{Q}(Z, V)=0$. Extend $Z$ to a vector field on $S^2$ whose support is contained in the region covered by the curves $\rho_s(\gamma)$, for $s\in (-\varepsilon,\varepsilon)$. Let $\zeta_{\overline{s}}$ denote the flow of $Z$.
	
The function $(s,\overline{s})\in(-\varepsilon,\varepsilon)\times (-\infty,+\infty) \mapsto L(\zeta_{\overline{s}}(\rho_s(\gamma)))$ is smooth, has a critical point at $(0,0)$, and, by the formula for the second variation of length (\ref{defi-quadratic-form-second-variation}), and our choice of vector field $Z$, the Hessian is negative definite at $(0,0)$. In particular, $(0,0)$ is an isolated maximum of the function $L(\zeta_{\overline{s}}(\rho_s(\gamma)))$. Therefore, decreasing $\delta<\varepsilon$ if necessary, we can find a continuous function  $\overline{s}: s\in [-\varepsilon, \varepsilon]\to [0,\eta]$ that vanishes on $[-\varepsilon, -\delta)\cup(\delta,\varepsilon]$, and satisfying $L(\zeta_{\overline{s}(s)}(\rho_s(\gamma))) < L(\gamma)$, for all $s \in [-\delta, \delta]$. It follows that, for all  $s\in[-\delta,\delta]$ and $p\in\zeta_{\overline{s}(s)}(\rho_{s}(\gamma))$, we have $dist(p,\zeta_{\overline{s}(s)}(\rho_{s}(x)))< dist(x,y)$. Replacing the map $f(t,s)$ from the proof of Proposition \ref{exis-sweep-closed-geod} by $(t,s)\mapsto \zeta_{\overline{s}(s)}(\rho_{s}(\gamma(t))$, and proceeding as in that argument, we obtain the sweepout with the desired properties.
\end{proof}

\begin{proof}[Proof of Theorem \ref{thm-index-bound}]
Let $\{x,y\}\subset S^2$ be a critical point with $dist(x,y)=W_d(S^2,g)>0$ as in the statement of the theorem. If the index of the closed geodesic $\gamma$ is not one, Proposition \ref{2-clo-stri} gives us a sweepout $\psi$ of $S^2$ such that $dist(\psi(p))<dist(x,y)$ for all $p\in S^2$. However, this contradicts the definition of $W_d(S^2,g)$ and completes the proof.
\end{proof}

\subsection{Proof of Theorem \ref{thm-comparison}}

In this section we apply the sweepouts produced in the proof of Theorem \ref{thm-3-simult-st} to prove the result about the direct comparison between our min-max invariant and the classic first width of Almgren-Pitts in the case of positively curved two-spheres.

\begin{proof}[Proof of Theorem \ref{thm-comparison}]
Calabi and Cao \cite{CalabiCao} proved that whenever $(S^2,g)$ has positive Gaussian curvature, the Almgre-Pitts width satisfies $\omega_1(S^2, g) = L(\gamma)$, for some simple closed geodesic $\gamma$. Fix $x$ and $y$ in $\gamma$ that divide this curve into two curves of the same length. Using the sweepout of Proposition \ref{exis-sweep-closed-geod} modified by Remark \ref{rmk-prop-generaliz-closed-geodesics}, we produce a sweepout $\psi\in[\Phi_0]$ such that 
$$
W_d(S^2, g) \leq \sup_{p\in S^2} dist(\psi(p)) \leq \frac{L(\gamma)}{2} = \frac{\omega_1(S^2, g)}{2}.
$$
In fact, the sweepout $\psi$ is of the form $\psi(p) = \{\psi_1(p), \psi_2(p)\}$, where $\psi_1(p)$ and $\psi_2(p)$ belong an element of a family of closed curves of length at most $L(\gamma)$. Moreover, $\gamma$ is the only curve of this family of length $L(\gamma)$. The estimate on the width comes from the fact that the distance between two points in a curve of length $\ell$ is at most $\ell/2$.

Moreover, if equality holds, then the sweepout above is optimal. By the facts explained in the preceding paragraph, $dist(\psi_1(p), \psi_2(p))=L(\gamma)/2$ can only hold if $\psi(p) = \{x,y\}$ is a pair of points dividing the closed geodesic $\gamma$ into curves of the same length. Moreover, if the two portions of $\gamma$ determined by $x$ and $y$ are not minimizing, then $\sup_{p\in S^2} dist(\psi(p)) < L(\gamma)/2$. Since this would contradict the optimality of $\psi$, we conclude that the two halves of that closed geodesic are minimizing. 
\end{proof}

\subsection{Convex geometry lemma}\label{sect.convex.geom} 

In this section we prove Lemma \ref{lemm-convex-geometry}.

\begin{proof}[Proof of Lemma \ref{lemm-convex-geometry}]
The hypotheses imply that the origin $0 \in \mathbb{R}^2$ belongs to the convex quadrilaterals $V(1,2,3,4)$, with vertices $v_1$, $v_2$, $v_3$, $v_4$ in this order, and $W(1,2,3,4)$, obtained by $w_1$, $w_2$, $w_3$, $w_4$ in this order. We also use the notation $V(i,j,k)$ to denote the triangle whose vertices are $v_i$, $v_j$, and $v_k$, and, similarly, use $W(i,j,k)$ for triangles with vertices in $\{w_i\}$. Using the diagonals of each quadrilateral, we verify that the origin belongs to a pair of triangles of type $V(i,j,k)$ with two adjacent vertices in common.

We proceed by contradiction. If the origin belongs to at least three different triangles of type $V(i, j, k)$ or of type $W(i,j,k)$, then there is nothing to be proved. Assume that $0$ belongs to exactly two of the triangles $V(i,j,k)$, and two of the $W(i,j,k)$ with different sets of indices. Suppose, without loss of generality, that the two adjacent vertices that belong to both triangles of type $V(i,j,k)$ containing the origin are $v_1$ and $v_2$. By the contradiction hypothesis, the two adjacent vertices that belong to both triangles of type $W(i,j,k)$ containing the origin are necessarily $w_3$ and $w_4$. Hence,
$$
0 \in V(1,2,3) \cap V(1,2,4) \cap W(1,3,4) \cap W(2,3,4).
$$
We assume that the origin belongs to the interior of those four triangles. Otherwise, the statement would hold. Figure \ref{quadrilateral} contains the quadrilateral of vertices $v_i$ in the unit circle, whose sides are represented by dashed segments, and the associated convex quadrilateral whose sides have lengths the $\lambda_i$. The contradiction hypothesis implies that $\theta_1 + \theta_2 > \pi$ and $\theta_1+\theta_4 >\pi$. 

\begin{figure}[h]
\includegraphics[scale=2.6]{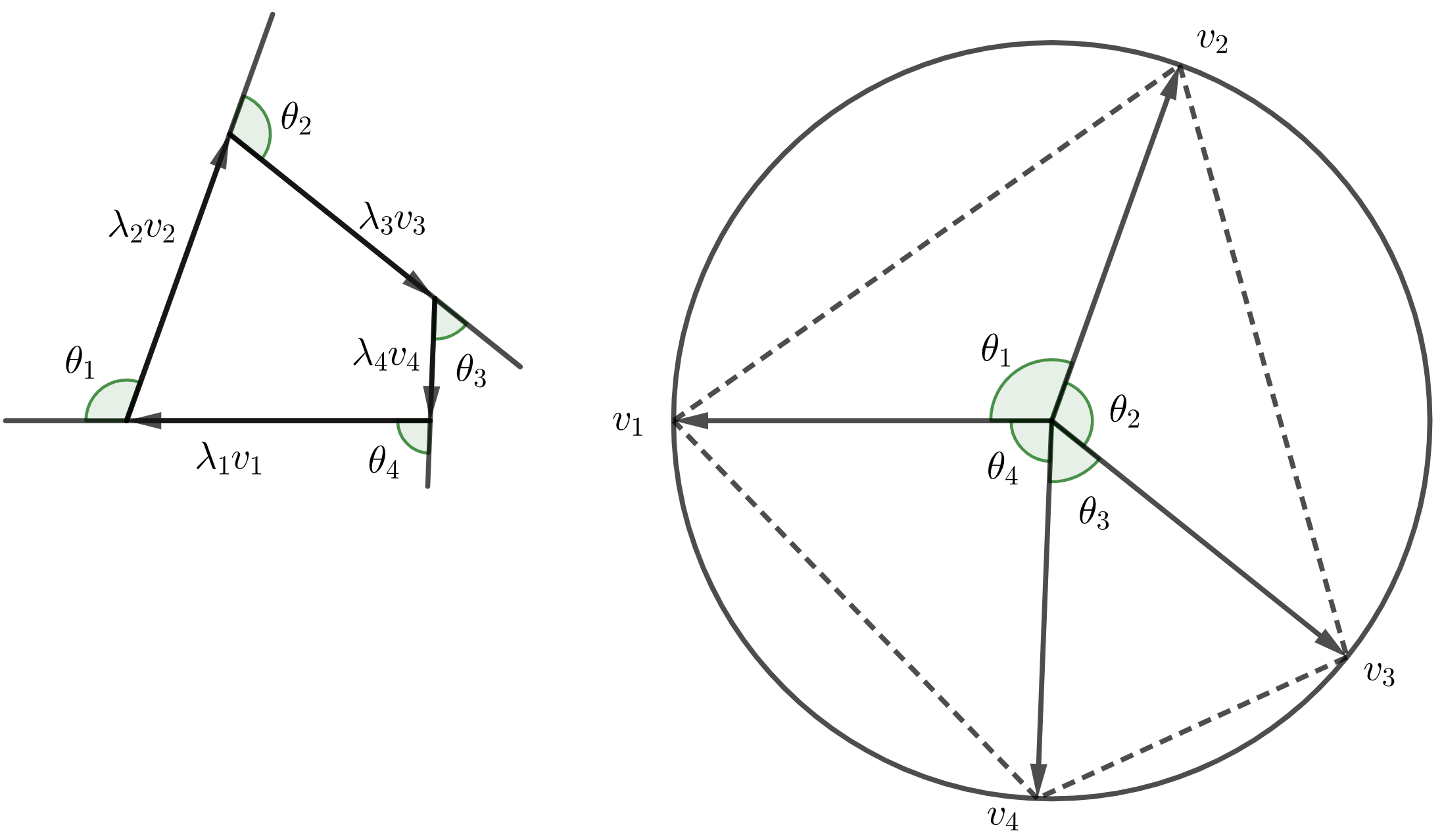}
\caption{Quadrilaterals associated to zero sum unit vectors with weights.}\label{quadrilateral}
\end{figure}

Similarly, let $\theta'_i$ be the corresponding central angles between the vectors $w_i$. More precisely, $\theta'_i$ is the measure of the angle determined by the rotation from $w_i$ to $w_{i+1}$ in the counterclockwise direction for $i=1,2,3$, and from $w_4$ to $w_1$ in the case of $\theta'_4$. The contradiction hypothesis yields $\theta'_2+\theta'_3 > \pi$ and $\theta'_3+\theta'_4>\pi$. Summarizing, one has $\min \{\theta_1 + \theta_2,  \theta_1+\theta_4, \theta'_2+\theta'_3, \theta'_3+\theta'_4 \} > \pi$.

The following claim allows us to express those four inequalities in terms of the angles $\theta:=\theta_4$, $\theta^{\prime}:=\theta'_4$, and the $\lambda_i$ only. Its proof, presented below, follows from elementary trigonometry arguments on the quadrilateral associated to the vanishing weighted sum of Figure \ref{quadrilateral}.

\textbf{Claim 1:} The contradiction hypothesis above, expressed in terms of $\min \{\theta_1 + \theta_2,  \theta_1+\theta_4, \theta'_2+\theta'_3, \theta'_3+\theta'_4 \} > \pi$, is equivalent to the following four inequalities involving the angles $\theta:=\theta_4$, $\theta^{\prime}:=\theta'_4$, and the $\lambda_i$ only:
\begin{align}\label{ineq-1}
(\lambda_4+\lambda_1\cos\theta)f(\theta) < \lambda_1\sin\theta (\lambda_1^2+\lambda_4^2+\lambda_2^2-\lambda_3^2+2\lambda_1\lambda_4\cos\theta),
\end{align}
\begin{align}\label{ineq-3}
(\lambda_1+\lambda_4\cos\theta')f(\theta') < \lambda_4\sin\theta' (\lambda_1^2+\lambda_4^2-\lambda_2^2+\lambda_3^2+2\lambda_1\lambda_4\cos\theta'),
\end{align}
\begin{align}\label{ineq-2}
\lambda_4\sin\theta (\lambda_1^2+\lambda_4^2-\lambda_2^2+\lambda_3^2+2\lambda_1\lambda_4\cos\theta) < (\lambda_1+\lambda_4\cos\theta)f(\theta),
\end{align}
\begin{align}\label{ineq-4}
\lambda_1\sin\theta' (\lambda_1^2+\lambda_4^2+\lambda_2^2-\lambda_3^2+2\lambda_1\lambda_4\cos\theta') < (\lambda_4+\lambda_1\cos\theta')f(\theta'),
\end{align}
where $f(x) := \sqrt{4\lambda_2^2\lambda_3^2 - (\lambda_1^2+\lambda_4^2-\lambda_2^2-\lambda_3^2+2\lambda_1\lambda_4 \cos x)^2}$.

\begin{proof}[Proof of Claim 1:]
Label the vertices of the associated quadrilateral $A$, $B$, $C$, and $D$, as in Figure \ref{lc-1}. The law of cosines on $ABD$ and $BCD$ gives
\begin{equation}\label{lc-4}
\lambda_1^2 + \lambda_4^2 + 2\lambda_1\lambda_4 \cos \theta_4 = \lambda_2^2 + \lambda_3^2 + 2\lambda_2\lambda_3 \cos \theta_2.
\end{equation}
Another applications of the law of cosines yields
\begin{equation}\label{lc-2}
\cos (\angle ABD) = \frac{\lambda_1+\lambda_4 \cos \theta_4}{|BD|}
\text{ and }
\sin (\angle ABD) = \frac{\lambda_4 \sin \theta_4}{|BD|},
\end{equation}
where $|BD|$ represents the length of the segment $BD$. Notice that the sign of $\sin (\angle ABD)$ is known to be non-negative because all $\theta_i$ measure at most $\pi$, by hypothesis, and $\theta_1 + \angle ABD \leq \pi$. Similarly, one has
\begin{equation}\label{lc-3}
\cos (\angle DBC) = \frac{\lambda_2+\lambda_3 \cos \theta_2}{|BD|}
\text{ and }
\sin (\angle DBC) = \frac{\lambda_3 \sin \theta_2}{|BD|}.
\end{equation}
Equations (\ref{lc-2}) and (\ref{lc-3}), and the fact that $\angle ABD + \angle DBC = \pi- \theta_1$ imply that
\begin{equation}
\sin \theta_1 = \frac{\lambda_4 \sin \theta_4(\lambda_2+\lambda_3 \cos \theta_2) + (\lambda_1+\lambda_4 \cos \theta_4)\lambda_3 \sin \theta_2 }{|BD|^2}
\end{equation}
and
\begin{equation}
\cos \theta_1 = \frac{\lambda_3 \lambda_4 \sin \theta_2 \sin \theta_4 - (\lambda_1+\lambda_4 \cos \theta_4)
(\lambda_2+\lambda_3 \cos \theta_2)}{|BD|^2}.
\end{equation}

\begin{figure}[h]
\includegraphics[scale=5]{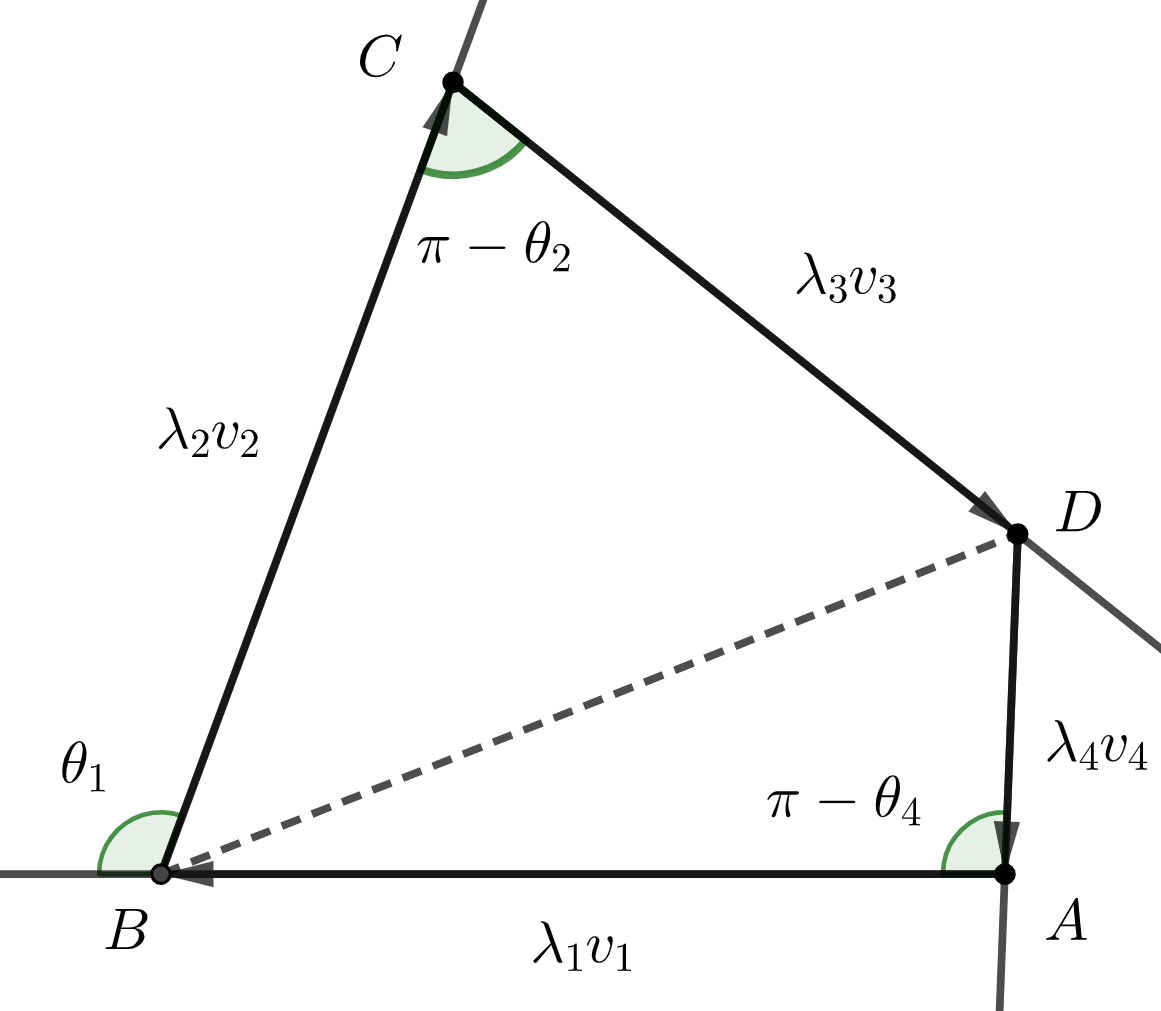}
\caption{Expressing $\theta_2$ in terms of $\theta_4$.} \label{lc-1}
\end{figure}

Equation (\ref{lc-4}) allows us to express $\theta_2$ in terms of the $\lambda_i$ and $\theta_4$:
\begin{equation*}
\sin \theta_2 = \frac{f(\theta_4)}{2\lambda_2\lambda_3} \text{ and } \cos \theta_2 = \frac{\lambda_1^2+\lambda_4^2-\lambda_2^2-\lambda_3^2+2\lambda_1\lambda_4 \cos \theta_4}{2\lambda_2\lambda_3},
\end{equation*}
where $f(\theta_4) := \sqrt{4\lambda_2^2\lambda_3^2 - (\lambda_1^2+\lambda_4^2-\lambda_2^2-\lambda_3^2+2\lambda_1\lambda_4 \cos \theta_4)^2}$ is the function defined in the statement of the claim.

Since $\theta_1+\theta_4 > \pi$ and $\theta_1+\theta_2 > \pi$, the values of $\sin(\theta_1+\theta_4)$ and $\sin(\theta_1+\theta_2)$ must be negative. Combining the above formulae, we obtain expressions for $\sin(\theta_1+\theta_4)$ and $\sin(\theta_1+\theta_2)$ in terms the $\lambda_i$ and $\theta_4$ only.  In conclusion, using $\theta:=\theta_4$, one obtains inequalities (\ref{ineq-1}) and (\ref{ineq-2}).

Similarly, for $\theta' :=\theta_4'$, one obtains inequalities (\ref{ineq-3}) and (\ref{ineq-4}). Notice that the function $f(\cdot)$ in the last two inequalities is the same as the previously defined function, since the relationship between the two inequalities (\ref{ineq-1}) and (\ref{ineq-2}) and the inequalities (\ref{ineq-3}) and (\ref{ineq-4}) is that one interchanges the roles of $\lambda_1$ and $\lambda_4$, and of $\lambda_2$ and $\lambda_3$. Therefore, the function $f$ is the same because it does not change after these interchanged positions.
\end{proof}

In order to conclude our contradiction argument, we must prove that the strict inequalities (\ref{ineq-1}), (\ref{ineq-2}), (\ref{ineq-3}), and (\ref{ineq-4}) can not hold simultaneously for a choice of angles $\theta$ and $\theta'$ in $(0,\pi)$. Otherwise, both of the equations
\begin{align}\label{eq-1-proof1}
(\lambda_4+\lambda_1\cos\alpha)f(\alpha) = \lambda_1\sin\alpha (\lambda_1^2+\lambda_4^2+\lambda_2^2-\lambda_3^2+2\lambda_1\lambda_4\cos\alpha),
\end{align}
and
\begin{align}\label{eq-2-proof1}
(\lambda_1+\lambda_4\cos\beta)f(\beta) = \lambda_4\sin\beta (\lambda_1^2+\lambda_4^2-\lambda_2^2+\lambda_3^2+2\lambda_1\lambda_4\cos\beta),
\end{align}
would have solutions. More precisely, for fixed $\lambda_1$, $\lambda_2$, $\lambda_3$, and $\lambda_4$, solutions $\alpha, \beta \in (0,\pi)$ for both equations would exist. Indeed, (\ref{ineq-1}) and (\ref{ineq-4}) would yield an $\alpha$ solving (\ref{eq-1-proof1}), and (\ref{ineq-2}) and (\ref{ineq-3}) would yield a $\beta$. This application of the Intermediate Value Theorem is possible because, regardless of the values of the $\lambda_i>0$, the domain of $f(\theta)$ is always an interval. Indeed, it is defined whenever
$$
-1\leq \frac{\lambda_1^2+\lambda_4^2-\lambda_2^2-\lambda_3^2+2\lambda_1\lambda_4 \cos \theta}{2\lambda_2\lambda_3} \leq 1.
$$

In what follows, we use $x = \cos\alpha$, $a=\lambda_1$, $b = \lambda_2$, $c=\lambda_3$, and $d=\lambda_4$. Squaring both sides of equation (\ref{eq-1-proof1}) gives us
\begin{align}\label{eq-1.3-proof1}
(d+a x)^2
\left(4b^2c^2 - (a^2+d^2-b^2-c^2+2ad x)^2\right)
\\ = a^2 (1-x^2) (a^2+d^2+b^2-c^2+2ad x)^2,\notag
\end{align}
Notice that the degree $4$ terms cancel out and the equation becomes a degree $3$ polynomial equation on $x$. The set of roots of (\ref{eq-1.3-proof1}) is given by
\begin{equation}\label{roots}
\left\{ -\frac{a^2+d^2}{2ad}, -\frac{a^2+d^2+b^2-c^2+2bd}{2a(b+d)}, \frac{a^2+d^2+b^2-c^2-2bd}{2a(b-d)}\right\},
\end{equation}
if $b\neq d$, $b\neq -d$, $a\neq 0$, and $d\neq 0$. Notice that, for the given values of $a$, $b$, $c$, and $d$, the conditions $b\neq -d$, $a\neq 0$, and $d\neq 0$ hold automatically. Direct substitution implies that the expressions in (\ref{roots}) are roots of (\ref{eq-1.3-proof1}). Analyzing the cases of equalities between each pair of such expressions yields that they are typically different, and, therefore, (\ref{roots}) contains all roots of (\ref{eq-1.3-proof1}), even when some of those coincide. In the case that $b=d$, the equation becomes a degree $2$ polynomial equation whose roots are in the set
\begin{equation}\label{roots-case-b=d}
\left\{ -\frac{a^2+b^2}{2ab}, -\frac{a^2+4b^2-c^2}{4ab}\right\}.
\end{equation}
When two of the roots in those sets coincide, they become multiple roots of the polynomial. In what follows, these two cases are analyzed separately. Recall that $x = \cos \alpha$, and $\alpha \in (0,\pi)$. Then, we are interested in roots $x \in (-1,1)$. The following claim will conclude the proof.

\textbf{Claim 2:} For fixed $\lambda_1=a$, $\lambda_2=b$, $\lambda_3=c$, and $\lambda_4=d$, if both of the equations (\ref{eq-1-proof1}) and (\ref{eq-2-proof1}) have solutions $\alpha, \beta \in (0,\pi)$, then $a=c$ and $b=d$. Moreover, in this case, expressions (\ref{eq-1-proof1}) and (\ref{eq-2-proof1}) become identities and the strict inequalities (\ref{ineq-1}), (\ref{ineq-2}), (\ref{ineq-3}), and (\ref{ineq-4}) can not hold simultaneously. 

\begin{proof}[Proof of Claim 2:]
We analyze two cases. 

\textbf{Case I.} Let us consider the case $b\neq d$. Let $x_1$, $x_2$, $x_3$ be the roots that appear in the set (\ref{roots}), in that order. Plugging in the root $x_2$ of (\ref{eq-1.3-proof1}) into the original equation, (\ref{eq-1-proof1}), and analyzing the signs of both sides of the equation yields that $x_2$ is a root of (\ref{eq-1-proof1}) if, and only if, $a^2 + b^2 = c^2 + d^2$. Moreover, this is equivalent to $x_2=x_3$. Therefore, it is not necessary to analyze the case of the root $x_2$ separately. 

In addition, one has $x_1 \leq -1$. Since we are only interested in roots strictly between $-1$ and $1$, this root should not be considered as well. In conclusion, if equation (\ref{eq-1-proof1}) has a root $x=\cos\alpha$, then
\begin{equation}\label{ineq-solve1}
x_3 = \frac{a^2+d^2+b^2-c^2-2bd}{2a(b-d)} = \frac{a^2-c^2}{2a(b-d)} + \frac{b-d}{2a} \in (-1,1).
\end{equation}
Similarly, if equation (\ref{eq-2-proof1}) has a root $x = \cos \beta \in (-1,1)$, then
\begin{equation}\label{ineq-solve2}
-1 < \frac{d^2-b^2}{2d(c-a)} + \frac{c-a}{2d} < 1.
\end{equation}
Suppose that (\ref{ineq-solve1}) holds, and let us prove the following fact.

\textbf{Claim 3:} Inequality (\ref{ineq-solve1}) implies that one of the following hold:
\begin{equation}\label{ineq-claim2-a1}
d<b,\, 0< 1/2-b < c, \text{ and } 1/2-b <a; \text{ or}
\end{equation}
\begin{equation}\label{ineq-claim2-a2}
d>b,\, 0< 1/2-a-b < c < 1/2-b.
\end{equation}

\begin{proof}[Proof of Claim 3:]
Indeed, if $d<b$, then the inequality in (\ref{ineq-solve1}) can be written in the form
\begin{equation}\label{ineq-pf-claim2-1}
(d-b-2a)(b-d) < a^2-c^2 < (d-b+2a)(b-d).
\end{equation}
Using $d = 1-a-b-c$, one has
\begin{align*}
(d-b-2a)(b-d) = a^2 - c^2 - (1-2a-2b)^2  +2c(1-2a-2b),
\end{align*}
and
\begin{align*}
(d-b+2a)(b-d) = a^2 -c^2 - (1-2b)^2  +2c(1-2b).
\end{align*}
After simplifications, the inequalities in (\ref{ineq-pf-claim2-1}) become $ 2c(1-2a-2b) < (1-2a-2b)^2$ and $(1-2b)^2 < 2c(1-2b)$. The second of these inequalities gives us that $0<1-2b < 2c$. We claim that $1-2a-2b<0$. Otherwise, $2c < 1-2a-2b$, which, together with $0<1-2b < 2c$, would imply $a<0$, contradicting the assumption $\lambda_i >0$, for all $i$. This concludes the proof of all relationships between the constants $a$, $b$, $c$, and $d$ stated in (\ref{ineq-claim2-a1}), assuming only that $d<b$.

If one assumes $d>b$ instead, then similar steps yield the other relations stated in (\ref{ineq-claim2-a2}). Here we omit the details.
\end{proof}

Let us state, without proof, the fact analog to Claim 3 for the case in which the inequalities described in (\ref{ineq-solve2}) hold.

\textbf{Claim 4:} Inequality (\ref{ineq-solve2}) implies that one of the following hold:
\begin{equation}\label{ineq-claim2-b1}
a<c,\, 0< 1/2-c < b, \text{ and } 1/2-c <d; \text{ or}
\end{equation}
\begin{equation}\label{ineq-claim2-b2}
a>c \text{ and } 0<1/2-d-c < b < 1/2-c.
\end{equation}

The goal is to prove that (\ref{ineq-solve1}) and (\ref{ineq-solve2}) can not hold simultaneously. Hence, using Claims 3 and 4, that is equivalent to show that each one of the conditions (\ref{ineq-claim2-a1}) and (\ref{ineq-claim2-a2}) is not compatible with each one of the conditions (\ref{ineq-claim2-b1}) and (\ref{ineq-claim2-b2}). Indeed, either directly or using that $a+b+c+d=1$, one can verify that both (\ref{ineq-claim2-a1}) and (\ref{ineq-claim2-b2}) imply $a+b>\frac{1}{2}$. Similarly, each one of the sets of relationships (\ref{ineq-claim2-a2}) and (\ref{ineq-claim2-b1}) implies $a+b <\frac{1}{2}$. Which yields that (\ref{ineq-claim2-a1}) and (\ref{ineq-claim2-b1}) are not compatible, and (\ref{ineq-claim2-a2}) and (\ref{ineq-claim2-b2}) are not compatible. Similarly, each one of the conditions (\ref{ineq-claim2-a1}) and (\ref{ineq-claim2-b1}) implies that $b+c >\frac{1}{2}$, and the relationships in either (\ref{ineq-claim2-a2}) or (\ref{ineq-claim2-b2}) imply $b+c <\frac{1}{2}$. This finishes the proof in the case $b\neq d$.

\textbf{Case II.} In case $b=d$, the polynomial equations in $x = \cos \alpha$ or $\cos\beta$ representing the squared versions of (\ref{eq-1-proof1}) and (\ref{eq-2-proof1}) are degree two polynomials. The first of these two equations has solutions given in (\ref{roots-case-b=d}). As in the case $b\neq d$, the first solution $x_1 = -(a^2+b^2)/(2ab) \leq -1$ is not suitable for us. Moreover, the second root of the squared equation solves the original equation, (\ref{eq-1-proof1}), if, and only if, $a=c$. Hence, (\ref{eq-1-proof1}) and (\ref{eq-2-proof1}) become identities which hold for every $\alpha$ and $\beta$. In particular, the strict inequalities (\ref{ineq-1}), (\ref{ineq-2}), (\ref{ineq-3}), and (\ref{ineq-4}) do not hold.
\end{proof}
This finishes the proof of Claim 2.
\end{proof}


\section{Constant width properties of \texorpdfstring{$W_d(\Sigma)$}{WdS} - Proofs}\label{sect-const-width-proofs}

In this section we prove the results stated in Section \ref{subsect-constant-width}.

\begin{proof}[Proof of Proposition \ref{prop-widths-coincide}]
Let us consider the case of strictly convex regions $\Omega$ whose boundary is smooth, and has strictly positive Gaussian curvature. The general case can be treated by approximation. Indeed, evolving a compact smooth convex surface by mean curvature flow instantaneously makes it strictly convex. See \cite{Andrews}, for instance. Also, the widths $W_d(\partial \Omega)$ and $w(\Omega)$ are continuous with respect to smooth convergence. Indeed, the difference between the classical widths of two regions is bounded by twice the Hausdorff distance between them. Hence, the difference decreases to zero under smooth (or even continuous) graphical convergence.

Let $\{\Omega_n\}$ be a sequence of compact regions with smooth boundary in $\mathbb{R}^3$, and $\Omega\subset \mathbb{R}^3$ be a compact region with smooth boundary. Suppose that the boundaries $\partial \Omega_n$ smoothly converge to $\partial \Omega$ in the sense that there exists a smooth map $X: \partial \Omega\times [0, T) \rightarrow \mathbb{R}^3$ such that $X|_{\partial \Omega \times \{0\}}$ is the identity map, and $X|_{\partial\Omega\times \{\delta_n\}}$ is an embedding whose image is $\partial \Omega_n$, for some sequence $\delta_n \rightarrow 0$. Given $\varepsilon>0$, let $\phi=\{\phi_1, \phi_2\}$ be a sweepout by pairs of points of $\partial \Omega$ for which $\max_{S^2} dist\circ \phi < W_d(\partial \Omega) +\varepsilon$. Then, $\phi$ clearly induces sweepouts $\psi^n$ by pairs of points of $\partial \Omega_n$; letting $\psi^n(x) = \{X(\phi_1(x), \delta_n), X(\phi_2(x), \delta_n)\}$, for $x\in S^2$. The smoothness of $X$ yields: $|dist(\psi^n(x)) - dist(\phi(x))|< \varepsilon$, for every $x\in S^2$, and large $n$. Then, $W_d(\partial \Omega_n) < W_d(\partial \Omega)+2\varepsilon$, for large $n$. Similarly, one can use sweepouts of the $\partial \Omega_n$ to induce sweepouts of $\partial \Omega$, proving that $W_d(\partial \Omega) < W_d(\partial \Omega_n)+2\varepsilon$, for large $n$.

One can parametrize the surface $\Sigma = \partial\Omega$ by the inverse Gauss map $\textbf{x}:= N^{-1} : S^2 \rightarrow \partial \Omega$, where $\textbf{x}(v) \in \partial \Omega$ is the unique boundary point for which the outward pointing unit normal vector is $v$, $N(\textbf{x}(v)) = v$, for all $v\in S^2$.

The support function $p: S^2 \rightarrow \mathbb{R}$, defined by $p(v):= \langle \textbf{x}(v), v\rangle$, can be used to represent the directional widths; $w(\Omega, v)= p(v) + p(-v)$, for all $v$. Letting $\left. \nabla p \right|_{v}$ denote the gradient of $p$ on the sphere computed at $v$, we claim that $\textbf{x}(v) = p(v) v + \left. \nabla p \right|_{v}$. It suffices to verify that $\langle \textbf{x}(v), w \rangle = \langle \left.\nabla p \right|_{v}, w \rangle$, for every $w$ orthogonal to $v$. Indeed, $\langle \left.\nabla p \right|_{v}, w \rangle = d p_v(w) = \langle d\textbf{x}_v(w), v\rangle + \langle \textbf{x}(v), w\rangle$. Finally, $\langle d\textbf{x}_v(w), v\rangle=0$ because $ d\textbf{x}_v$ maps $T_vS^2$ to $T_{\textbf{x}(v)}(\partial \Omega)$, and both of these are orthogonal to $v$.

The classical width $w(\Omega)$ is realized by some directional width $w(\Omega, v_0)$. Then, $0 = \nabla w|_{v_0} =  \nabla p|_{v_0} - \nabla p|_{-v_0}$. This implies that
$$
\textbf{x}(v_0) - \textbf{x}(-v_0) = (p(v_0)+p(-v_0)) v_0 + \left. \nabla p \right|_{v_0} - \left. \nabla p \right|_{-v_0} = w(\Omega, v_0) v_0.
$$
In particular, the line segment from $\textbf{x}(v_0)$ to $\textbf{x}(-v_0)$ is an orthogonal chord of length $w(\Omega)$. Since orthogonal chords realize corresponding directional widths for convex $\Omega$, $w(\Omega)$ is the length of the shortest orthogonal chord.

In order to finish the proof, we prove that $W_d(\partial \Omega)$ is also equal to the length of the shortest orthogonal chord. Theorem \ref{thm-W_d-critical} tells us that $W_d(\partial \Omega)= dist(x_0, y_0)$ is a positive critical value of the distance, where $\{x_0, y_0\} \in \mathcal{P}_{\Sigma}$ is a critical point. On $\mathbb{R}^3$, any two points are joined by a single geodesic. Therefore, the minimizing geodesic connecting $x_0$ and $y_0$ is an orthogonal chord. It remains to show that given any orthogonal chord $c = [x_1, x_2]$ of $\Omega$, connecting $x_1, x_2 \in \partial \Omega$, there exists a sweepout $\phi : S^2 \rightarrow \Sigma$ such that 
$$
\max_{x \in S^2} dist(\phi(x)) = L(c) = dist(x_1, x_2).
$$ 

Let $v^{\ast} \in S^2$ be the unit vector such that $x_1-x_2$ is a positive multiple of $v^{\ast}$. Letting $N(x)$ denote the outward pointing unit normal vector to $\Sigma=\partial \Omega$ at $x$, the convexity hypothesis on $\Omega$ implies that $N: \partial \Omega \rightarrow S^2$ is a diffeomorphism. In particular, the regions defined by 
$$\Sigma^{+} = \{x \in \Sigma : \langle N(x), v^{\ast}\rangle \geq 0\} \text{ and } \Sigma^{-} = \{x \in \Sigma : \langle N(x), v^{\ast}\rangle \leq 0\}$$
are images via $N^{-1}$ of the closed hemispheres determined by $\langle w, v^{\ast} \rangle = 0$.

In the construction of the desired sweepout $\phi$, we apply the auxiliary map $F : \Sigma \rightarrow \Sigma$ defined by $F(x)=x$, if $x \in \Sigma^{-}$, and, if $x$ belongs to the interior of $\Sigma^{+}$ in $\Sigma=\partial\Omega$, $F(x)$ is the point in $\Sigma$ such that $F(x) = x + t v^{\ast}$, for some $t\neq 0$. Geometrically, $x$ and $F(x)$ are the points of intersection of $\Sigma$ and the line parallel to $v^{\ast}$ passing through $x$, and $F(x) = x+tv^{\ast}$, for some $t<0$. 

The map $F$ is continuous on the interior of $\Sigma^{+}$ in $\Sigma$, and it extends continuously to the boundary as the identity map. Indeed, if $\{x_n\}$ is a sequence of points in $\Sigma^{+}\setminus(\Sigma^{+}\cap\Sigma^{-})$, which converges to $x \in \Sigma^{+}\cap\Sigma^{-}$, then, $F(x_n) = x_n + t_n v^{\ast}$, for some $t_n<0$, and for every $n\geq 1$. Suppose, by contradiction, that $F(x_n)$ does not converge to $x$. Then, a subsequence, which we still denote $\{x_n\}$ and $\{t_n\}$, is such that there exists $\varepsilon>0$ such that $t_n<-\varepsilon$, for all $n$. The compactness of $\Sigma$ tells us that $\{t_n\}$ is bounded. We can assume, without loss of generality, that $t_n$ converges to some $t<0$. Hence, $F(x_n)$ tends to $x + tv^{\ast}$, and this is a contradiction to the convexity assumption. Indeed, on one hand, the limit $x + tv^{\ast}$ must be in $\Sigma$, as a limit of points on this surface. On the other hand, the convexity implies that $x$ is the only point of intersection of $\Sigma$ and $T_x\Sigma$. But $x + tv^{\ast} \in T_x\Sigma\setminus\{x\}$, because $v^{\ast} \in T_x\Sigma$ for every $x \in \Sigma^+\cap\Sigma^-$.

In conclusion, $F$ is continuous on $\Sigma$. Moreover, notice that $F(\Sigma)\subset \Sigma^{-}$, which implies that $F$ is homotopic to a constant map. We are now ready to define the sweepout $\phi$. We use $\Sigma$ as the domain of the sweepout, which is diffeomorphic to $S^2$. Define $\phi(x)=\{x, F(x)\}$, for every $x\in \Sigma$. The map $\phi$ is continuous and homotopic to $\Phi_0$. Hence, it is a sweepout satisfying $dist(\phi(x)) = dist(x, F(x)) \leq dist(x_1, x_2)$, for every $x \in \Sigma$. The convexity assumption implies that $\Sigma$ is contained in the region bounded by the tangent planes $T_{x_1}\Sigma$ and $T_{x_2}\Sigma$, which gives us the upper bound. Moreover, $F(x_1)=x_2$ implies that equality holds for $x=x_1$. 
\end{proof}

\begin{proof}[Proof of Theorem \ref{thm-Reidemeister}]
Let us prove part (I). Assume that $W_d(\Sigma)$ equals the (extrinsic) diameter of $\Sigma$. Suppose, by contradiction, that there exists $x_1 \in \Sigma$ such that $dist (x_1,y) < W_d(\Sigma)$, for every $y \in \Sigma$. Let $i : S^2 \rightarrow M$ be an embedding with $i(S^2)= \Sigma$, and $x_0 \in S^2$ be such that $i(x_0) = x_1$. Consider $\phi : S^2 \rightarrow \mathcal{P}_\Sigma$ defined by $\phi(x) := \{i(x), i(x_0)\}$, for all $x \in S^2$. Notice that $\phi$ is homotopic to $\Phi_0$. Indeed, $\Phi_0$ is defined similarly, the only difference being that $\Phi_0(x) = \{i(x), i(x_2)\}$, for every $x\in S^2$, where $x_2 \in S^2$ is, possibly, different from $x_0$. These two maps are homotopic. Indeed, letting $c(t)$ be a path in $S^2$ joining $c(0)=x_0$ and $c(1)=x_2$, $\phi_t(x):=\{ i(x), i(c(t))\}$ is such that $\phi_0 = \phi$ and $\phi_1 = \Phi_0$. Therefore, $\phi$ is a sweepout whose supremum of the distances $dist(\phi(x))$, $x\in S^2$, equals $\sup_{y \in \Sigma} dist (x_1,y) < W_d(\Sigma)$. This contradicts Definition \ref{defi-width} and proves property (I).

Let us prove (II). Recall that $\Psi$ is a continuous map defined on $\Sigma$ such that $\Psi(x) \in \Sigma$ and $dist(x, \Psi(x)) = diam(\Sigma)$, for every $x\in \Sigma$. Let $\phi : S^2 \rightarrow \mathcal{P}_{\Sigma}$ be an arbitrary sweepout of $\Sigma$, as in Definition \ref{defi-sweepout}, such that $\phi(x) = \{\phi_1(x), \phi_2(x)\}$, for a
pair of continuous maps $\phi_1, \phi_2 : S^2\rightarrow \Sigma$. We claim that there exists $x\in S^2$ and $y \in \Sigma$ such that $\phi(x) = \{y, \Psi(y)\}$. More precisely, we prove that $\phi_2(x)=\Psi(\phi_1(x))$ for some $x\in S^2$.

Suppose, by contradiction, that this is not the case. Then, both points $\phi_1(x)$ and $\phi_2(x)$ do not coincide with $\Psi(\phi_1(x))$, for every $x \in S^2$. Notice that $\phi_1(x)\neq \Psi(\phi_1(x))$ always holds because the distance between these two points equals the diameter of $\Sigma$, which is a positive number. Let $i : S^2\rightarrow \Sigma$ be a fixed embedding from the $2$-sphere to $\Sigma$. For every $x\in S^2$, let us use $\pi_x : S^2 \setminus \{ i^{-1}\circ \Psi \circ \phi_1 (x)\}  \rightarrow \mathbb{R}^2$ to denote the stereographic projection of $S^2$ with base point $i^{-1}\circ \Psi \circ \phi_1 (x)$. Let $\alpha : S^2 \times [0,1] \rightarrow \Sigma$ be the continuous map defined, for $x\in S^2$ and $t\in [0,1]$, by
$$
\alpha(x, t) : = i\circ \pi_x^{-1} \circ \left( (1-t) \cdot \left(\pi_x\circ i^{-1} \circ \phi_1(x)\right) + t \cdot \left(\pi_x\circ i^{-1} \circ \phi_2(x)\right)\right).
$$
The map is well defined because of the contradiction hypothesis. For fixed $x$, the map simply composes $i$ with a parametrization of the arc of the circle through $i^{-1}(\phi_1(x))$, $i^{-1}(\phi_2(x))$, and $i^{-1}(\Psi\circ \phi_1(x))$, joining the first two of these three points and not contain the third, whenever $\phi_1(x)\neq \phi_2(x)$. Otherwise, $\alpha(x,t)$ is constant with respect to $t$. Therefore, this construction yields that $\phi(x) = \{\phi_1(x), \alpha(x, 1)\}$ is homotopic to $\{\phi_1(x), \alpha(x, 0)\} = \{\phi_1(x)\}$. This is a contradiction and the proof of (II) is finished.

To prove the final part of the statement of the theorem, we show that, on the intrinsic case, $\Sigma = M$, every point $x \in \Sigma$ belongs to at most one pair $\{x,y\}$ which is a critical point of the distance function with $y\neq x$. This property holds for any metric. Indeed, if $\{x,y\}$ is critical, Proposition \ref{prop-characterization-critical-simult-st} implies that there exists an integer $k \leq 5$, and a set of $k$ simultaneously stationary minimizing geodesics $\gamma_j$, $1\leq j\leq k$, with endpoints $x$ and $y$. These geodesics meet at $x$ and $y$ only, because of the fact that they are minimizing. The complement in $\Sigma$ of the union of these curves has $k$ connected components denoted by $U_j$, $1\leq j \leq k$.

Suppose, by contradiction, that there exists $z \notin \{x,y\}$ such that $\{x,z\}$ is also a critical point of the distance. Since each $\gamma_j$ is minimizing from $x$ to $y$, it is the only minimizing geodesic from $x$ to any point in $\gamma_j\setminus\{y\}$. In particular, $z\notin \gamma_j$. Assume, for simplicity, that $z\in U_1$. Every minimizing geodesic $\alpha$ joining $x$ and $z$ must be completely contained in $U_1$. Otherwise, it would cross some $\gamma_j$ which is part of the boundary of $U_1$, contradicting the minimizing properties of $\alpha$ and $\gamma_j$. Definition \ref{defi-simultaneously-stationary} implies that the angles at $x$ and $y$ between two consecutive $\gamma_k$ and $\gamma_l$ is at most $\pi$. In particular, all velocity vectors of minimizing geodesics $\alpha$ as above would be contained in an angle measuring strictly less that $\pi$, at both $x$ and $y$. This contradicts the fact that $\{x, z\}$ is critical.

In order to finish the proof, observe first that the forward part of the final claim in the statement is a particular case of (I). Conversely, assume that for each $x \in \Sigma$ there exists $y \in \Sigma$, such that $dist(x,y)= diam(\Sigma)$. These pairs $\{x,y\}$ are, in particular, the only non-trivial critical points of the distance function. This is a consequence of the claim that we proved in the preceding paragraphs. Hence, $diam(\Sigma)$ is the unique positive critical value of the distance function. Theorem \ref{thm-W_d-critical} implies that $W_d(\Sigma) = diam(\Sigma)$.
\end{proof}

\begin{proof}[Proof of Theorem \ref{thm-involution}]
The fixed point theorem of Brouwer implies that the isometric involution $\Psi : \Omega \rightarrow \Omega$ has a fixed point $p_0 \in \Omega$. By assumption, $p_0\notin \partial \Omega =: \Sigma$. We claim that $p_0$ is unique. Assume, by contradiction, that $\Psi(p_i)=p_i$, for $p_0\neq p_1$. Let $\gamma$ be the minimizing geodesic parametrized by arc length such that $\gamma(0) = p_0$ and $\gamma(a) = p_1$. Since $\Psi$ is an isometry, its image $\Psi\circ \gamma$ is also a minimizing geodesic joining the same endpoints. By the uniqueness hypothesis (i) in the statement of the theorem, we conclude that $\gamma = \Psi \circ \gamma$. This implies that $\gamma$ is made of fixed points of $\Psi$ only. This fact and hypothesis (ii) in the statement of the theorem contradict the hypothesis on the non-existence of boundary fixed points.

We claim that $d \Psi_{p_0} = - Id_{T_{p_0} M}$. Since $d \Psi_{p_0} \circ d \Psi_{p_0} = Id_{T_{p_0} M}$, the possible eigenvalues of $d \Psi_{p_0}$ are $\pm 1$. As before, hypothesis (ii) implies that $+1$ is not an eigenvalue of that map, since the geodesic emanating from $p_0$ in the direction of the corresponding eigenvector would be made of fixed points of $\Psi$. The dimension of $T_{p_0}M$ is three, then, $d \Psi_{p_0}$ has $-1$ as an eigenvalue. Let $v\neq 0$ be such that $d \Psi_{p_0}(v) = -v$. The map $d \Psi_{p_0}$ is an isometry, therefore, the orthogonal complement $\langle v\rangle^{\perp}$ of the line spanned by $v$ is invariant by $d \Psi_{p_0}$. We analyze two possibilities: the restriction of $d \Psi_{p_0}$ to $\langle v\rangle^{\perp}$ is either a rotation, or the composition of a reflection with a rotation. The latter does not occur, since a reflection composed with a rotation always has the number one as an eigenvalue. Then, $d \Psi_{p_0}$ is a rotation on $\langle v\rangle^{\perp}$ and the angle of rotation must necessarily be equal to $\pi$, because of the involutive property of the map $\Psi$. Therefore, the claim is proved.

Let $z_0 \in \Sigma=\partial \Omega$ be such that $dist(p_0, z_0) = dist(p_0, \Sigma)$, and let $\beta$ be a minimizing geodesic parametrized by arc length and joining the points $\beta(0)=p_0$ and $\beta(a)=z_0$. The choice of $z_0$ clearly implies, by the first variation formula, that $\beta$ is orthogonal to $\Sigma$ at $z_0$, and is the only minimizing geodesic joining $p_0$ and $z_0$. Since $d\Psi_{p_0} = -Id_{T_{p_0}M}$, $\beta\cup \Psi(\beta)$ is a (smooth) geodesic connecting $z_0$ and $\Psi(z_0)$. 

The fact that $\beta$ is minimizing and meets $\Sigma$ orthogonally implies that the same holds for $\Psi(\beta)$, and, in particular, these two curves meet at $p_0$ only. Moreover, $\beta$ being minimizing implies that $L(\beta\cup\Psi(\beta)) \leq 2\cdot diam(\Omega)$. Therefore, hypothesis (iv) and $\Psi$ being an isometry imply that $\beta\cup\Psi(\beta)$ is minimizing and satisfies the free boundary condition. It follows that $z_0$ and $\Psi(z_0)$ form a critical point of the distance function on $\mathcal{P}_{\Sigma}$. 

Next, we claim that every point of $\Sigma$ belongs to a unique non-trivial critical point of the distance function defined on $\mathcal{P}_{\Sigma}$. Let $\{x,y\}$ be such a critical point. Since $x, y \in \Sigma = \partial \Omega$, hypothesis (i) gives us that there exists a unique minimizing geodesic connecting $x$ and $y$. Being a critical point, it follows that this minimizing geodesic must satisfy the free boundary condition. A unique free boundary geodesic can emanate from either $x$ or $y$ and stay in $\Omega$, by the convexity assumption included in hypothesis (i). This implies that $x$ does not belong to another non-trivial critical point of the distance. Moreover, the main assumption, $W_d(\Sigma) = diam(\Sigma)$, implies, via part (I) of Theorem \ref{thm-Reidemeister}, that every point $x \in \Sigma$ belongs to a critical point $\{x, y\}$ that is in the level $W_d(\Sigma)$ of the distance function. 

By the fact proved in the preceding paragraph, it follows that 
\begin{align}\label{eqtn-proof-invol-1}
 2\cdot dist(p_0, z_0) = 2 \cdot L(\beta) = L(\beta\cup\Psi(\beta))  = W_d(\Sigma) = diam(\Sigma)
\end{align}

Given $z\in \Sigma$, let $\alpha : [0, \ell] \rightarrow M$ be the minimizing geodesic parametrized by arc length with $\alpha(0)=p_0$ and $\alpha(\ell) = z$. The minimizing curves $\alpha$ and $\Psi(\alpha)$ emanating from $p_0$ only intersect at $p_0$, and, by hypothesis (iv), $\alpha\cup \Psi(\alpha)$ is a minimizing geodesic joining $z$ and $\Psi(z)$. In particular, 
\begin{align}\label{eqtn-proof-invol-2}
2\cdot dist(p_0, z) = 2\cdot L(\alpha) = dist(z, \Psi(z)) \leq diam(\Sigma).
\end{align}
Combining (\ref{eqtn-proof-invol-1}), (\ref{eqtn-proof-invol-2}), and the choice of $z_0$, it follows, for every $z \in \Sigma$, that $dist(p_0, z) = dist(p_0, \Sigma)$. The remaining steps of the proof are exactly as in the last four paragraphs in the proof of Theorem F in \cite{AMS} to conclude that $\Omega$ is a metric and geodesic ball in $(M^3, g)$ of radius $W_d(\Sigma)/2$.
\end{proof}

\textbf{Data Availability:} Data sharing not applicable to this article as no datasets were generated or analyzed during the current study.

\textbf{AI usage statement:} We did not make use of LLMs to conduct the research that led to this article. This paper does not contain AI-generated text.


\begin{thebibliography}{99}

\bibitem{Amb2013} L. Ambrozio, \textit{Rigidity of area-minimizing free boundary surfaces in mean convex three-manifolds}, Journal of Geometric Analysis, Volume 25 (2015), Issue 2, pp 1001-1017.

\bibitem{AmbMon} L. Ambrozio and R. Montezuma, \textit{On the min-max width of unit volume three-spheres}, Journal of Differential Geometry, 126 (2024), Issue 3, p. 875-907.

\bibitem{AMS} L. Ambrozio, R. Montezuma, and R. Santos, \textit{The width of embedded circle.} (accepted for publication at J. Reine Angew. Math.)  arXiv:2307.12939v2 [math.DG].

\bibitem{Andrews} B. Andrews, J. McCoy, and Y. Zheng, \textit{Contracting convex hypersurfaces by curvature.} Calc. Var. 47, 611–665 (2013). https://doi.org/10.1007/s00526-012-0530-3

\bibitem{Bal} F. Balacheff. \textit{A local optimal diastolic inequality on the two-sphere,}  J. Topol. Anal. 2
(2010), no. 1, 109-121.

\bibitem{Beach} I. Beach and R. Rotman. \textit{The Length of the Shortest Closed Geodesic on a Surface of Finite Area.} Proceedings of the American Mathematical Society. 2020. Vol. 148 (12), p. 5355--5367.

\bibitem{CalabiCao} E. Calabi and J. G. Cao, \textit{Simple closed geodesics on convex surfaces.}  J. Differential Geom. 36(3): 517-549 (1992). DOI: 10.4310/jdg/1214453180

\bibitem{ChaGro} G.D. Chakerian and H. Groemer, 
\textit{Convex bodies of constant width.} Convexity and its applications, 49-96. Birkhäuser Verlag, Basel, 1983.

\bibitem{ChoMan1} O. Chodosh and C. Mantoulidis, \textit{Minimal surfaces and the Allen–Cahn equation on 3-manifolds: index, multiplicity, and curvature estimates.} Annals of Mathematics, Volume 191 (2020), Issue 1, 213-328.

\bibitem{ChoMan} O. Chodosh and C. Mantoulidis, \textit{The p-widths of a surface.} Publ. Math. Inst. Hautes Etudes Sci. 137 (2023), 245–342.

\bibitem{CC}  C. Croke, 
\textit{Area and length of the shortest closed geodesic}, J. Differential Geometry 27 (1988), 1-21.

\bibitem{BonFen} T. Bonnesen and W. Fenchel, \textit{Theorie der konvexen Körper.} Springer,Berlin (1934)(Reproduced by Chelsea Publ. Co., New York, English translation: Theory of Convex Bodies, Translated from the German and edited by Boron, L., Christenson, C., Smith, B. BCS Associates, Moscow, ID, Boron, 1971), p. 1987.

\bibitem{Don} S. Donato. \textit{Donato, S. The First p-Widths of the Unit Disk.} J Geom Anal 32, 177 (2022). https://doi.org/10.1007/s12220-022-00913-3

\bibitem{DM} Donato, S., Montezuma, R. \textit{The First Width of Non-negatively Curved Surfaces with Convex Boundary.}  J Geom Anal 34, 60 (2024). https://doi.org/10.1007/s12220-023-01511-7

\bibitem{Hat} A. Hatcher \textit{Algebraic topology.} Cambridge University Press, Cambridge, 2002, xii+544 pp.

\bibitem{LioMarNev} Y. Liokumovich, F. Cod\'a Marques, and A. Neves, \textit{Weyl law for the volume spectrum,} Annals of Mathematics 187 (2018), no. 3, 933-961.

\bibitem{SarStr} J. Marx-Kuo, L. Sarnataro, and D Stryker, \textit{Index, Intersections, and Multiplicity of Min-Max Geodesics.} (preprint) 	arXiv:2410.02580 [math.DG]

\bibitem{ManSch} C. Mantoulidis and R. Schoen, \textit{On the Bartnik mass of apparent horizons,} Class. Quantum Grav. 32 2015, no. 20, 205002.

\bibitem{MarNevDuke} F. Cod\'a Marques and A. Neves, \textit{Rigidity of min-max minimal spheres in three-manifolds.}
Duke Math. J. 161 (2012), no. 14, 2725-2752.

\bibitem{MarNev} F. Cod\'a Marques and A. Neves, \textit{Min-max theory and the Willmore conjecture.} Annals of
Mathematics (2014), 683-782.

\bibitem{MarMonOli} H. Martini, L. Montejano, and D. Oliveros, \textit{Bodies of constant width. An introduction to convex geometry with applications.} Birkhäuser/Springer, Cham, 2019. xi+486 pp.

\bibitem{Pitts} J. T. Pitts, \textit{Regularity and singularity of one dimensional stationary integral varifolds on manifolds arising from variational methods in the large.} Symposia Mathematica, Vol. XIV (Convegno di Geometria Simplettica e Fisica Matematica \& Convegno di Teoria Geometrica dell’Integrazione e Variet\`a Minimali, INDAM, Rome, 1973), Istituto Nazionale di Alta Matematica, Rome, 1974, pp. 465–472.

\bibitem{Rei} K. Reidemeister, \textit{Über Körper konstanten Durchmessers.} Math.Z. 10, 214-216 (1921).

\bibitem{SS} S. Sabourau, \textit{Local extremality of the Calabi-Croke sphere for the length of the shortest closed geodesic}. Journal of the London Mathematical Society, vol. 82 (2010), no. 3, 549-562.

\bibitem{Aiex} N. Sarquis Aiex, \textit{The Width of Ellipsoids.} Comm. in Analysis and Geometry, Vol.27 (2019), no.2, pp.251-285 

\bibitem{Song} A. Song, \textit{Existence of infinitely many minimal hypersurfaces in closed manifolds.} Annals of Mathematics, Volume 197 (2023), Issue 3, 859-895.

\bibitem{Zhou} X. Zhou, \textit{On the Multiplicity One Conjecture in min-max theory.} Annals of Mathematics, Volume 192 (2020), Issue 3, 767-820.

\end{thebibliography}
\end{document}